\documentclass[a4paper,11pt]{article}
\usepackage[margin = 1in]{geometry}
\usepackage[utf8]{inputenc}
\usepackage[T1]{fontenc}
\usepackage{amssymb}
\usepackage{nccmath}
\usepackage{geometry}
\usepackage{latexsym}
\usepackage{amsmath}
\usepackage{graphics}
\usepackage{fullpage}
\usepackage{epsfig}
\usepackage{amsthm}
\usepackage{relsize}
\usepackage{tikz}
\usepackage{amsfonts}
\usepackage{mathrsfs}
\usepackage{makeidx}
\usepackage{comment}
\usepackage{latexsym}
\usepackage{dcolumn}
\usepackage{tocloft} 
\usepackage{calc}
\usepackage{pgfplots}
\usepackage{pgfplotstable}
\usepackage{makecell}
\usepackage{mathtools}
\usepackage{xfrac}
\usepackage{mathtools}
\usepackage[
     colorlinks = true,
     citecolor = blue,
     urlcolor = blue
     ]{hyperref}
\usepackage[ocgcolorlinks]{ocgx2}
\usepackage[shortlabels]{enumitem}
\excludecomment{codes}
\usepackage[strict = true,
            style = english]{csquotes}
\DeclareTextCommand{\polhk}{OT4}{\k}
\DeclareTextCommand{\polhk}{T1}{\k}
\usepackage[backend = biber,
						style = numeric,
						natbib = true,
            sorting = nyt,
            abbreviate = true,
           ]{biblatex}
\AtEveryBibitem{
 \clearfield{doi}
 \clearfield{url}
 \clearfield{month}
 \clearfield{series}

 \ifentrytype{book}{}{
  \clearlist{publisher}
  \clearname{editor}
 }
}

\theoremstyle{plain}
\newtheorem{theorem}{Theorem}[section]
\newtheorem{corollary}[theorem]{Corollary}
\newtheorem{lemma}[theorem]{Lemma}
\newtheorem{proposition}[theorem]{Proposition}

\theoremstyle{definition}
\newtheorem{definition}[theorem]{Definition}

\newtheorem{remark}[theorem]{Remark}

\newtheorem{example}[theorem]{Example}

\numberwithin{equation}{section}
\allowdisplaybreaks

\newcommand{\R}{\mathbb{R}}
\newcommand{\E}{\mathbb{E}}

\newcommand{\Norm}[1]{\left\| #1 \right\|}

\newcommand{\ud}{\ensuremath{\mathrm{d}}}

\title{Invariant measures for the nonlinear stochastic heat equation with no drift term}
\author{
  Le Chen           \\Auburn University\\\url{le.chen@auburn.edu}\\ \and
  Nicholas Eisenberg\\Auburn University\\\url{nze0019@auburn.edu}
  }
\date{\today}

\begin{document}
\maketitle

\begin{abstract}
  This paper deals with the long term behavior of the solution to the nonlinear stochastic heat
  equation $\partial u /\partial t - \frac{1}{2}\Delta u = b(u)\dot{W}$, where $b$ is assumed to be
  a globally Lipschitz continuous function and the noise $\dot{W}$ is a centered and spatially
  homogeneous Gaussian noise that is white in time. Using the moment formulas obtained
  in~\cite{chen.huang:19:comparison,chen.kim:19:nonlinear}, we identify a set of conditions on the
  initial data, the correlation measure and the weight function $\rho$, which will together
  guarantee the existence of an invariant measure in the weighted space $L^2_\rho(\R^d)$. In
  particular, our result includes the \textit{parabolic Anderson model} (i.e., the case when $b(u) =
  \lambda u$) starting from the Dirac delta measure.
\end{abstract}

\smallskip

\smallskip

 \smallskip

\setcounter{tocdepth}{2}
{
  \hypersetup{linkcolor=black}
  \tableofcontents
}

\section{Introduction}
In this paper, we study the following \textit{nonlinear stochastic heat equation} (SHE):
\begin{equation}\label{E:SHE}
    \begin{cases}
     \dfrac{\partial u}{\partial t}(t,x)-\dfrac{1}{2}\Delta u(t,x) = b(x,u(t,x))\dot W(t,x) & \text{$x\in \R^d$, $t>0$}, \\
     u(0,\cdot) = \mu(\cdot).
   \end{cases}
\end{equation}
The noise, $\dot{W}(t,x)$, is a centered Gaussian noise that is white in time and homogeneously
colored in space defined on a complete probability space $(\Omega, \mathcal{F}, \mathbb{P})$ with
the natural filtration $\{\mathcal{F}_t \}_{t\ge 0}$ generated by the noise. Its covariance
structure is given by
\begin{align}\label{E:Cor}
  J(\psi,\phi) \coloneqq \E \left[ W\left(\psi\right) W\left(\phi\right) \right]
  = \int_0^\infty \ud s \int_{\R^d} \Gamma(\ud x)(\psi(s,\cdot)*\widetilde{\phi}(s,\cdot))(x),
\end{align}
where $\psi$ and $\phi$ are continuous and rapidly decreasing functions, $\widetilde{\phi}(x)
\coloneqq \phi(-x)$, ``$*$'' refers to the convolution in the spatial variable, and $\Gamma$ is a
nonnegative and nonnegative definite tempered measure on $\R^d$ that is commonly referred to as the
\textit{correlation measure}. The Fourier transform of $\Gamma$ (in the generalized sense) is also a
nonnegative and nonnegative definite tempered measure, which is usually called the \textit{spectral
measure} and is denoted by $\widehat{f}(\ud \xi)$ (see~\eqref{E:Fourier} for the convention of
Fourier transform). Moreover, in the case where $\Gamma$ has a density $f$, namely, $\Gamma(\ud x) =
f(x) \ud x$, we write $\widehat{f}(\ud \xi)$ as $\widehat{f}(\xi) \ud \xi$.

The initial condition, $\mu$, is a deterministic, locally finite, regular, signed Borel measure that
satisfies the following integrability condition at infinity:
\begin{equation}\label{A:icon}
  \int_{\R^d}\exp\left(-a|x|^2\right) |\mu|(\ud x) < \infty \quad \text{for all $a>0$},
\end{equation}
where $|\mu| = \mu_++\mu_-$ and $\mu = \mu_+-\mu_-$ refers to the \textit{Hahn decomposition} of the
measure $\mu$. Initial conditions of this type, introduced in~\cite{chen.dalang:15:moments*1} and
further explored in~\cite{chen.kim:19:nonlinear,chen.huang:19:comparison}, are called \textit{rough
initial conditions}.

The function $b(x,u)$ is uniformly bounded in the first variable and Lipschitz continuous in the
second variable, i.e., for some constants $L_b>0$ and $L_0\ge 0$,
\begin{equation}\label{E:lipcon}
  |b(x,u)-b(x,v)| \le L_b |u-v| \quad \text{and} \quad
  |b(x,0)|        \le L_0       \quad \text{for all $u,v \in \R$ and $x \in \R^d.$ }
\end{equation}
In particular, our assumption allows the linear case $b(x,u) = \lambda u$, which is usually referred
to as the \textit{parabolic Anderson model} (PAM)~\cite{carmona.molchanov:94:parabolic}.

The SPDE~\eqref{E:SHE} is understood in its \textit{mild form}:
\begin{equation}\label{E:gmsol}
  u(t,x) = J_0(t,x;\mu)+ \int_0^t\int_{\R^d}b(y,u(s,y))G(t-s,x-y)W(\ud s,\ud y),
\end{equation}
where $G(t,x) = (2 \pi t)^{-d/2}\exp\left(-\left(2t\right)^{-1}|x|^2\right)$ is the heat kernel,
\begin{gather}\label{E:Ju}
   J_0(t,x) = J_0(t,x;\mu) \coloneqq (G(t,\cdot) * \mu)(x) = \int_{\R^d} G(t,x-y) \mu(\ud y)
\end{gather}
is the solution to the homogeneous equation, and the stochastic integral is the \textit{Walsh
integral}. We refer the interested readers
to~\cite{chen.kim:19:nonlinear,dalang.khoshnevisan.ea:09:minicourse,dalang:99:extending,walsh:86:introduction}
for more details of this setup. \bigskip

The aim of this paper is to investigate the conditions required to guarantee the existence of an
invariant measure for the solution to~\eqref{E:SHE}, which is a crucial step for the study of the
ergodicity of the system that requires the corresponding uniqueness. We direct the interested
readers to~\cite{cerrai:01:second,da-prato.zabczyk:96:ergodicity, da-prato.zabczyk:14:stochastic}
for more details about the invariant measure, its existence/uniqueness, and the ergodicity of the
system. The general procedure for finding the invariant measure, especially in the setting
of~\eqref{E:SHE}, has been laid out by Tessitore and Zabczyk~\cite{tessitore.zabczyk:98:invariant},
which involves two parts: first one needs to show that the laws of the solution to~\eqref{E:SHE}
form a family of \emph{Markovian transition functions} on some Hilbert space, $H$, and the
corresponding \emph{Markovian semigroup} is Feller; and second one needs to establish that the
moments of solution are bounded in time (see~\eqref{E:BddMnt} below). For the second point, it
requires some substantial work (see Theorems~\ref{T:Lrho} and~\ref{T:Main} below). On the other
hand, the first point has been shown to be true for our case of interest (see, e.g.,~\cite[Chapter
9]{da-prato.zabczyk:14:stochastic}) with the following weighted $L^2 (\R^d)$ space as our underlying
Hilbert space as in~\cite{tessitore.zabczyk:98:invariant}:

\begin{definition}[\cite{tessitore.zabczyk:98:invariant}]\label{D:Rho}
  A function $\rho:\R^d\mapsto \R$ is called an \textit{admissible weight function} if it is a
  strictly positive, bounded, continuous, and $L^1 (\R^d)$-integrable function such that for all
  $T>0$, there exists a constant $C_\rho(T)$ such that
  \begin{equation}\label{E:aw}
    \big(G(t,\cdot)*\rho(\cdot)\big)(x) \le C_\rho(T) \rho(x)
    \quad \text{for all $t\in[0,T]$ and $x\in\R^d$.}
  \end{equation}
  Moreover, we denote by $L^2_{\rho}(\R^d)$ the corresponding Hilbert space of $\rho$-weighted
  square integrable functions, and we use $\langle \cdot, \cdot \rangle_{\rho}$ and
  $\Norm{\cdot}_{\rho}$ to denote the inner product and norm on $L^2_{\rho}(\R^d)$:
  \begin{align*}
  		\langle f,g \rangle_{\rho} \coloneqq \int_{\R^d} f(x) g(x) \rho(x) \ud x \quad \text{and} \quad
      \Norm{f}_{\rho} \coloneqq \int_{\R^d} |f(x)|^2 \rho(x) \ud x.
  \end{align*}
\end{definition}

Accordingly, we will prove the existence of the invariant measure following the same strategy as
in~\cite{tessitore.zabczyk:98:invariant}. Let $\mathscr{L}(u(t,\cdot;\mu)$ denote the law of
$u(t,\cdot)$ starting from $\mu$ at $t = 0$. We will first establish the tightness of
$\{\mathscr{L}(u(t,\cdot;\mu)\}_{t>t_0}$ for some $ t_0\ge 0$. A critical step in obtaining this
tightness result is to show that the following moment is uniformly bounded in time (see
Theorem~\ref{T:Main}):
\begin{align}\label{E:BddMnt}
  \sup_{t>0} \E\left(\Norm{u(t,\cdot)}_{\rho}^2\right) <\infty.
\end{align}
Then we will apply the \textit{Krylov-Bogoliubov theorem} (see,
e.g.,~\cite[Theorem~11.7]{da-prato.zabczyk:14:stochastic}) to construct an invariant measure via
\begin{equation}\label{E:InvMeasForm}
  \eta(A) = \lim_{n \to \infty} \frac{1}{T_n} \int_{t_0}^{T_n + t_0} \mathscr{L}(u(t,\cdot; \mu))(A) \ud t,
\end{equation}
for some sequence $\{T_n\}_{n\ge 1}$ with $T_n\uparrow \infty$. \medskip

In the literature, the existence of invariant measure of the stochastic heat equation is more
commonly studied with a drift term; we will postpone a brief review of this case to
Section~\ref{SS:Drift}. In contrast, the existence of an invariant measure under the settings of
equation~\eqref{E:SHE} has rarely been studied. To the best of our knowledge, this current article
and the one by Tessitore and Zabczyk~\cite{tessitore.zabczyk:98:invariant} are the only papers that
consider the case where the spatial domain is the whole space $\R^d$, the diffusion term, $b(x,u)$,
is globally Lipschitz in the second variable, uniformly bounded in the first variable
(see~\eqref{E:lipcon}) and there is no additional negative drift term to help. The major challenge
is to identify the right conditions so that the probability moments of the solution are bounded in
time (see~\eqref{E:BddMnt}). The solution to~\eqref{E:SHE} is usually \textit{intermittent}, namely,
its moments possess a certain exponential growth in $t$; see,
e.g.,~\cite{carmona.molchanov:94:parabolic,foondun.khoshnevisan:09:intermittence}. For that reason,
one has to impose some additional assumptions either on the initial conditions, or the noise, or the
coefficients of~\eqref{E:SHE}, or all of them in order to control the growth of the moments. The
moment formulas obtained in~\cite{chen.huang:19:comparison,chen.kim:19:nonlinear} play an important
role in this context.

Here we emphasize that we study the invariant measure using the Walsh random field
approach~\cite{walsh:86:introduction}, whereas such studies are mostly carried out under the
framework of the stochastic evolution in Hilbert spaces~\cite{da-prato.zabczyk:14:stochastic}. Even
though both theories are equivalent (see~\cite{dalang.quer-sardanyons:11:stochastic}), the
differences in many technical aspects are still substantial. As the random field approach often
produces results that are more explicit, we try to use this approach to obtain more precise
conditions for the existence of an invariant measure. For the initial conditions, the results
in~\cite{tessitore.zabczyk:98:invariant} allow for bounded $L_\rho^2 (\R^d)$ functions, although the
authors proved their main result---Theorem 3.3 \textit{ibid.}---only for the constant one initial
condition. Here we give the precise conditions on the initial condition (see~\eqref{E:InitData}
below) which allows a wider class of data, including unbounded functions and measures such as the
Dirac delta measure (see Examples~\ref{Ex:delta} and~\ref{Ex:Riesz}). Note that the Dirac delta
initial measure plays a prominent role in the study of the stochastic heat equation; see,
e.g.,~\cite{amir.corwin.ea:11:probability}. Regarding the noise, we give an explicit and easily
verifiable condition---\eqref{E:Dalang-00}---on the spectral density $\widehat{f}$ and present a few
concrete examples (see Section~\ref{SS:Bassel}). The comparisons of our conditions with those
obtained by Tessitore and Zabyczyk~\cite{tessitore.zabczyk:98:invariant} are given in
Section~\ref{SS:TessZabz}.


Our proof relies on a factorization representation for the solution $u(t,x)$ to~\eqref{E:SHE} (see
Lemma~\ref{L:Factorize}), which is obtained under the random field framework, whereas such
factorization lemma has been widely used in the framework of the stochastic evolution equation in
Hilbert spaces; see Section~\ref{S:Factor} for more details. Finally, we point out that there is a
miscellany of results in Section~\ref{S:Example}, which may have independent interest. \medskip

Now we are ready to motivate the conditions that we use and present the main results.

\subsection{Main results}\label{S:Main}

As mentioned earlier, in order to have moments bounded in time as in~\eqref{E:BddMnt}, one should
better first identify the sharp conditions under which the second moments as a function of $t$, namely
$t \mapsto \E\left(u(t,x)^2\right)$, with $x$ fixed, are bounded. This question has been answered
in~\cite[Theorem 1.3 and Lemma 2.5]{chen.kim:19:nonlinear}, where necessary and sufficient
conditions are given. More precisely, to have the second moment bounded in time with $x$ fixed, one
needs to have the spatial dimension $d \ge 3$, and in addition, the spectral measure $\hat{f}$ and
Lipschitz constant $L_b$ of $b(\cdot)$ need to satisfy the following two conditions:
\begin{subequations}\label{E:U0finite}
\begin{gather}\label{E:Dalang-00}
  \Upsilon(0) \coloneqq \left(2\pi\right)^{-d} \int_{\R^d} \frac{\hat{f}(\ud \xi)}{|\xi|^2} <\infty \shortintertext{and}
  64 L^2_{b} < \frac{1}{2\Upsilon(0)} \:. \label{E:Lipb}
\end{gather}
\end{subequations}
These two conditions will guarantee the existence of the following non-empty open interval:
\begin{align}\label{E:Interval}
\left(2^7 L_b^2 \Upsilon(0), 1\right) \ne \emptyset.
\end{align}

Note that condition~\eqref{E:Dalang-00} is a strengthened version of \textit{Dalang's condition}:
\begin{align}\label{E:Dalang}
  \Upsilon(\beta) \coloneqq \left(2\pi\right)^{-d} \int_{\R^d} \frac{\hat{f}(\ud \xi)}{\beta+ |\xi|^2} <\infty,
  \quad \text{for some (and hence) all $\beta>0$.}
\end{align}
Recall that in order to obtain the H\"older continuity of the solution, one needs to
strengthen~\eqref{E:Dalang} in a different way. Indeed, what is required is that for some $\alpha\in
(0,1]$,
\begin{align}\label{E:Dalang-ab}
  \Upsilon_\alpha(\beta) \coloneqq \left(2\pi\right)^{-d} \int_{\R^d} \frac{\widehat{f}(\ud\xi)}{\left(\beta+|\xi|^2\right)^{1-\alpha}}<\infty
  \quad \text{for some (hence all) $\beta>0$;}
\end{align}
see~\cite[Theorem 1.8]{chen.huang:19:comparison} or~\cite{sanz-sole.sarra:02:holder}. Likewise, one
can further strengthen condition~\eqref{E:Dalang-ab} to
\begin{equation}\label{E:Dalang-a0}
  \Upsilon_\alpha(0) \coloneqq \left(2\pi\right)^{-d} \int_{\R^d} \frac{\hat{f}(\ud \xi)}{|\xi|^{2(1-\alpha)}} <\infty
  \quad \text{for some $\alpha\in (0,1]$.}
\end{equation}
We use the convention that when $\alpha = 0$, we simply drop it from the expression
$\Upsilon_\alpha(\beta)$, i.e., $\Upsilon(\beta) = \Upsilon_0(\beta)$. The relations of these
conditions are illustrated in Figure~\ref{F:Dalang}.


\setlength{\abovecaptionskip}{0pt plus 0pt minus 2pt}
\begin{figure}[htpb]
\begin{center}
  \usetikzlibrary{arrows.meta}
  \begin{tikzpicture}[scale = .8, x = 9em, y = 6em]
		\newcommand\inc{0.07}
		\newcommand\xo{1.2}
		\newcommand\yo{0.5}
		\newcommand\xg{3.7}
		\newcommand\yg{1.5}



		\node[text width = 13em, align = center] (01) at (\xo,\yg) {$\Upsilon(\beta)<\infty$\\ Dalang's condition \eqref{E:Dalang}};
		\node[text width = 13em, align = center] (00) at (\xo,\yo) {$\Upsilon(0)<\infty$\\ Condition \eqref{E:Dalang-00}};
		\node[text width = 13em, align = center] (10) at (\xg,\yo) {$\Upsilon_\alpha(0)<\infty$\\ Condition \eqref{E:Dalang-a0}};
		\node[text width = 13em, align = center] (11) at (\xg,\yg) {$\Upsilon_\alpha(\beta)<\infty$\\ Condition \eqref{E:Dalang-ab}};
		\node[] at (\xo,0.5*\yo+0.5*\yg) {\Large$\Uparrow$};
		\node[] at (\xg,0.5*\yo+0.5*\yg) {\Large$\Uparrow$};
		\node[] at (0.5*\xo+0.5*\xg,\yg) {\Large$\Leftarrow$};

\end{tikzpicture}
\end{center}
  \caption{Relations among conditions~\eqref{E:Dalang},~\eqref{E:Dalang-ab},~\eqref{E:Dalang-a0}
  and~\eqref{E:Dalang-00}.}
	\label{F:Dalang}
\end{figure}

We will also need the following slightly different condition:
\begin{gather}\label{E:Dalang-H}
  \mathcal{H}_{\alpha / 2} (t) < \infty \quad \text{for some $\alpha\in (0,1]$ and for all $t>0$,} \shortintertext{where}
  \mathcal{H}_{\alpha}(t) \coloneqq \int_0^t \ud r \: r^{-2\alpha} \int_{\R^d} \hat{f}(\ud \xi) \: \exp(-r|\xi|^2). \label{E:H}
\end{gather}
The quantity $\mathcal{H}_{\alpha}(t)$ will appear naturally in the proof of Lemma~\ref{L:yineq}
below. As shown in Lemma~\ref{L:HUA} below, condition~\eqref{E:Dalang-a0} will imply
condition~\eqref{E:Dalang-H}. However, if one assumes~\eqref{E:Dalang-00}, then these two conditions
become equivalent.

We are now ready to state our two main results of the paper.

\begin{theorem}\label{T:Lrho}
  Let $u(t,x;\mu)$ be the solution to~\eqref{E:SHE} starting from $\mu$ which
  satisfies~\eqref{A:icon}. Assume that
  \begin{enumerate}[(i)]
    \item $\rho: \R^d \to \R_+$ is a nonnegative $L^1(\R^d)$ function;
    \item for all $t > 0$, the initial condition $\mu$ satisfies $\mathcal{G}_{\rho}(t;|\mu|) <
      \infty$ where
      \begin{equation}\label{E:G}
        \mathcal{G}_{\rho}(t;\mu)
        \coloneqq \int_{\R^d} J_0^2\left(t,x;\mu\right) \rho(x)\: \ud x;
      \end{equation}
    \item the spectral measure $\widehat{f}$ and the Lipschitz constant $L_b$ satisfy the two
      conditions in~\eqref{E:U0finite}.
  \end{enumerate}
  Then there exists a unique $L^2 (\Omega)$-continuous solution $u(t,x)$ such that for some constant
  $C>0$, which does not depend on $t$, the following holds:
  \begin{equation}\label{E:uInLrho}
    \E\left(\Norm{u(t,\cdot; \mu)}_{\rho}^2\right)
    \le C \mathcal{G}_{\rho}(t; \mu^*)
    < \infty, \quad \text{for any $t > 0$},
  \end{equation}
  where $ \mu^* \coloneqq 1 + |\mu|$.
\end{theorem}

This theorem will be proved in Section~\ref{S:Lrho}. We now state and prove a corollary which shows
that the solution to~\eqref{E:SHE} starting from an $L^2_{\rho}(\R^d)$ initial condition will almost
surely be in $L^2_{\rho}(\R^d)$ for all $t >0$.

\begin{corollary}\label{C:inclusion}
  Under the same assumptions of Theorem~\ref{T:Lrho}, if in addition $\rho$ is admissible (see
  Definition~\ref{D:Rho}), then the solution $u(t,\cdot;\zeta)$ to $\eqref{E:SHE}$ almost surely
  exists in $L^2_{\rho}(\R^d)$ for all $t>0$, whenever the initial condition is also in
  $L^2_{\rho}(\R^d)$, i.e., $\zeta\in L^2_{\rho}(\R^d)$.
\end{corollary}
\begin{proof}
  Choose and fix an arbitrary $\zeta \in L^2_{\rho}(\R^d)$ and set $\zeta^* = 1+|\zeta|$. It is clear
  that $\zeta^*\in L^2_\rho(\R^d)$. By (ii) of Theorem~\ref{T:Lrho}, it suffices to show the
  finiteness of $\mathcal{G}_{\rho}(t,;\zeta^*)$ for all $t>0$. Indeed, by H\"older's inequality,
  \begin{align*}
        \mathcal{G}_{\rho}(t,;\zeta^*)
     = \int_{\R^d} \left(\int_{\R^d} G(t,x-y) \zeta^* (y) \ud y\right)^2 \rho(x) \ud x
    \le \int_{\R^d} \int_{\R^d} G(t,x-y) \zeta^* (y)^2 \rho(x) \ud y \ud x.
  \end{align*}
  Now for any $t>0$, choose $T>t$ and let $C_{\rho}(T)$ be as in~\eqref{E:aw}. Then,
  \begin{align*}
      \int_{\R^d} \int_{\R^d} G(t,x-y) \zeta^* (y)^2 \rho(x) \ud y \ud x
    = \int_{\R^d} \left(G(t,\cdot)*\rho\right)(y) |\zeta(y)|^2 \ud y
    \le C_{\rho}(T) \Norm{\zeta^*}_{\rho}^2 < \infty.
  \end{align*}
  Since $t$ and $T$ are arbitrary, this completes the proof of the corollary.
\end{proof}

\begin{theorem}\label{T:Main}
  Let $u(t,x)$ be the solution to~\eqref{E:SHE} starting from $\mu$ and let $\rho$ be an admissible
  weight function. Assume that
  \begin{enumerate}[(i)]
    \item there exists another admissible weight $\tilde{\rho}$ such that
      \begin{align}\label{E:rhorho}
        \int_{\R^d}\dfrac{\rho(x)}{\tilde{\rho}(x)}\ud x < \infty;
      \end{align}
    \item the weight function $\tilde{\rho}$ and the initial condition satisfy the following
      condition:
      \begin{equation}\label{E:InitData}
        \limsup_{t>0}\: \mathcal{G}_{\tilde{\rho}}(t;|\mu|) < \infty;
      \end{equation}
    \item the spectral measure $\widehat{f}$ and the Lipschitz constant $L_b$ satisfy the two
      conditions in~\eqref{E:U0finite};
    \item for some $\alpha\in \left(2^7\Upsilon(0)L_b^2,1\right)$ (see~\eqref{E:Interval}), the
      spectral measure $\widehat{f}$ satisfies~\eqref{E:Dalang-a0}.
  \end{enumerate}
  Then we have that
  \begin{enumerate}[(a)]
    \item for any $\tau>0$, the sequence of laws of $\{\mathcal{L}u(t,\cdot;\mu)\}_{t\ge \tau}$ is
      tight, i.e., for any $\epsilon \in (0,1)$, there exists a compact set $\mathscr{K}\subset
      L^2_{\rho}(\R^d)$ such that
      \begin{equation}\label{E:cpt}
        \mathscr{L}u(t,\cdot;\mu)(\mathcal{K}) \ge 1-\epsilon,
        \qquad \text{for all $t\ge \tau>0$};
      \end{equation}
    \item there exists an invariant measure for the laws $\{\mathcal{L}u(t,\cdot;\mu)\}_{t > 0}$ in
      $L^2_\rho(\R^d)$.
  \end{enumerate}
\end{theorem}

This theorem will be proved in Section~\ref{S:mainproof}.

\subsection{Outline and notation}

The paper is organized as follows: we first prove Theorem~\ref{T:Lrho} in Section~\ref{S:Lrho}. Then
in Section~\ref{S:Factor}, we study the factorization lemma. Then we proceed to prove
Theorem~\ref{T:Main} in Section~\ref{S:mainproof}. Finally, in Section~\ref{S:Example} we make some
further discussion on the main results and present various examples. In particular, in
Section~\ref{SS:Drift}, we give a brief review of the problem of finding invariant measures for the SHE
with a drift term; in Section~\ref{SS:TessZabz}, we compare our conditions on the spectral density
with those obtained by Tessitore and Zabczyk~\cite{tessitore.zabczyk:98:invariant}; in
Section~\ref{SS:Init}, we show that our results could include a wider class of initial conditions;
in Section~\ref{SS:Bassel}, we carry out some explicit computations for the Bessel and related
kernels as the correlation functions; finally, in Section~\ref{SS:Weight}, we give a few examples of
the admissible weight functions.

\bigskip

We conclude this Introduction by introducing some notation and formulas that we use throughout the
paper. We will use $\Norm{X}_p$ to denote the $L^p (\Omega)$ norm, namely,
$\Norm{X}_p = \left(\mathbb{E}(|X|^p)\right)^{1/p}$. We will also use the following factorization
property of the heat kernel,
\begin{equation}\label{E:gsim}
  G(t,x)G(s,y) = G\left(\frac{ts}{t+s}, \frac{sx + ty}{t+s}\right)G\left(t+s, x-y\right),
\end{equation}
which can be easily verified and has been used extensively and critically
in~\cite{chen.dalang:15:moments*1,chen.kim:19:nonlinear,chen.huang:19:comparison}. Next, we will
need the following spherical coordinate integration formula:
\begin{align*}
  \int_{\R^d} f(|x|) \ud x = \sigma(\mathbb{S}^{d-1}) \int_0^\infty f(r) r^{d-1} \ud r,
\end{align*}
where $\sigma(\mathbb{S}^{d-1}) = 2\pi^{d/2} / \Gamma(d/2)$ and $\Gamma(x)$ denotes the Gamma
function. We use ``$\sim$'' to denote the standard asymptotic equivalent relation. Lastly, the
convention of Fourier transform is given by (see Remark~\ref{R:Fourier})
\begin{align}\label{E:Fourier}
  \widehat{\phi}(\xi) = \mathcal{F}\phi(\xi) \coloneqq \int_{\R^d} e^{-i x\cdot\xi} \phi(x)\ud x \quad \text{and} \quad
  \mathcal{F}^{-1}\psi(x) \coloneqq (2\pi)^{-d} \int_{\R^d} e^{i x\cdot\xi} \psi(\xi)\ud \xi.
\end{align}

\section{Moment estimates -- Proof of Theorem~\ref{T:Lrho}}\label{S:Lrho}
We first state some known results and prove a moment bound in Corollary~\ref{C:Mom}.

\begin{theorem}[Theorem 1.2 of~\cite{chen.huang:19:comparison}]\label{T:Solve}
  Suppose that
  \begin{enumerate}[(i)]
   \item the initial deterministic measure $\mu$ satisfies~\eqref{A:icon};
   \item the spectral measure $\widehat{f}$ satisfies \textit{Dalang's condition}~\eqref{E:Dalang},
  \end{enumerate}
  Then~\eqref{E:SHE} has a unique random field solution starting from $\mu$. Moreover, the solution
  is $L^2 (\Omega)$ continuous and is adapted to the filtration $\{\mathcal{F}_t\}_{t\ge 0}$.
\end{theorem}

\begin{theorem}[Theorem 1.7 of~\cite{chen.huang:19:comparison}]\label{T:ExUn}
  Under the assumptions of Theorem~\ref{T:Solve}, for any $t>0$, $x \in \R^d$ and $p \ge 2$, the
  solution to~\eqref{E:SHE}, $u(t,x)$, given by~\eqref{E:gmsol} is in $L^p (\Omega)$ and
  \begin{equation}\label{E:mb}
    \Norm{u(t,x)}_p \le \big[\bar{\varsigma}+\sqrt{2}(G(t,\cdot) * |\mu|)(x)\big]H(t;\gamma_p)^{1/2},
  \end{equation}
  where $\bar{\varsigma} = L_0/L_b$, $\gamma_p = 32pL_b^2$ (see~\eqref{E:lipcon} for $L_0$ and $L_b$)
  and the function $H(t;\gamma_p)$ is nondecreasing in $t$ (see~\cite{chen.huang:19:comparison} for
  the expression of the function $H$).
\end{theorem}

\begin{corollary}\label{C:Mom}
  Under the same setting as Theorem~\ref{T:ExUn}, if the two conditions in~\eqref{E:U0finite} hold
  (see also~\eqref{E:Interval}), then
  \begin{equation}\label{E:mb0}
    \Norm{u(t,x)}_p \le C_p\bigg(1+ (G(t, \cdot)*|\mu|)(x) \bigg),
    \quad \text{for all $p$ such that $1/p \in \left(64 L_b^2\Upsilon(0),1\right)$},
  \end{equation}
  where $C_p = \left(\sqrt{2}\vee \overline{\varsigma}\right)\sup_{t\ge 0}
  H(t;\gamma_p)^{1/2}<\infty$.
\end{corollary}
\begin{proof}
  Lemma 2.5 of~\cite{chen.kim:19:nonlinear} gives one sufficient condition, namely
  $2\gamma_p\Upsilon(0)<1$, under which the function $H((t;\gamma_p)$ is bounded in $t$. Therefore,
  by taking into account the expression of $\gamma_p$ in Theorem~\ref{T:ExUn}, we see that as a
  direct consequence of~\eqref{E:mb}, whenever
  \begin{equation}\label{E:pLip}
     32 p L_b^2 < \frac{1}{2\Upsilon(0)},
  \end{equation}
  we have the $p$-th moment bounded as given in~\eqref{E:mb0}.
\end{proof}

Now we are ready to prove Theorem~\ref{T:Lrho}.

\begin{proof}[Proof of Theorem~\ref{T:Lrho}]
  Under condition (iii), we can apply Fubini's Theorem and the moment bound~\eqref{E:mb0} below to
  see that for some constant $C>0$ independent of $t$, which may vary from line to line, that
  \begin{align*}
    \E\left(\Norm{u(t,\cdot;\mu)}_{\rho}^2\right)
    & \le C\: \E \left[\int_{\R^d}\bigg(1+(G(t,\cdot) * |\mu|)(x)\bigg)^2\rho(x)\ud x\right] \\
    & = C \int_{\R^d}\E \left[\bigg(\Big(G(t,\cdot) *(1+|\mu|)\Big)(x) \bigg)^2\right]\rho(x)\ud x \\
    & = C \:\mathcal{G}_{\rho}(t;\mu^*) < \infty,
  \end{align*}
  where we recall that $\mu^* = 1+|\mu|$. This proves Theorem~\ref{T:Lrho}.
\end{proof}

\begin{remark}[Restarted SHE]\label{R:Restat}
  Recall that the Markov property of the solution to~\eqref{E:SHE} implies that for any $t\ge
  t_0>0$,
  \begin{align}\label{E:eqLaw}
    u(t + t_0,x;\mu) \stackrel{\mathcal{L}}{=}
    u\left(t,x;u\left(t_0,\cdot;\mu\right)\right) \eqqcolon v(t,x),
  \end{align}
  where $\mathcal{L}$ refers to the equality in law. Then $v$ satisfies the following restarted
  SPDE:
  \begin{equation}\label{E:restart}
    \begin{cases}
      \dfrac{\partial v}{\partial t}(t,x)-\dfrac{1}{2}\Delta v(t,x) = b(x,v(t,x))\dot W_{t_0}(t,x) & \text{$x\in \R^d$ , $t>0$}, \\
      v(0,x) = u(t_0,x;\mu),                                                                       & x\in\R^d,
    \end{cases}
  \end{equation}
  where $\dot{W}_{t_0}(t,x) \coloneqq \dot{W}(t + t_0,x)$ denotes the time shifted noise, i.e.,
  \begin{align}\label{E:ShiftW}
    \int_0^t \int_{\R^d} W_{t_0}(\ud s, \ud y) = \int_{t_0}^{t+t_0} \int_{\R^d} W(\ud s, \ud y).
  \end{align}
  Under the conditions in~\eqref{E:U0finite}, Theorem~\ref{T:ExUn} and~\eqref{E:eqLaw} imply
  immediately that
  \begin{align*}
  \Norm{v(t,x)}_q = \Norm{u(t+t_0,x;\mu)}_q
                  \le C_q\bigg(1 + (G(t+t_0, \cdot)* |\mu|)(x) \bigg)
                  = C_q J_0(t+t_0,x; 1+|\mu|),
  \end{align*}
  for all $q\ge 2$ and $t>0$, where the constant $C_q$ does not depend on $t$. Moreover, under the
  assumptions of Theorem~\ref{T:Lrho}, we have $v(0,\cdot)\in L_\rho^2 (\R^d)$ a.s. and
  \begin{align*}
        \E\left(\Norm{v(t,x)}_\rho^2\right)
    = \E\left(\Norm{u(t+t_0,x;\mu)}_\rho^2\right)
    \le C \mathcal{G}_\rho(t+t_0;\mu^*)
    < \infty.
  \end{align*}
\end{remark}

\section{A factorization lemma}\label{S:Factor}

In this section, we establish a factorization lemma with corresponding moment estimates; see
Lemmas~\ref{L:yineq} and~\ref{L:Factorize} below. This factorization lemma appeared
in~\cite{da-prato.kwapien.ea:87:regularity}; check also Section~5.3.1
of~\cite{da-prato.zabczyk:14:stochastic}. For $\alpha\in (0,1)$, $t>0$ and $x\in\R^d$, define
formally
\begin{align}\label{E:Ops}
  \left(F_{\alpha} f\right)(t,x) & \coloneqq \int_0^{t} \int_{\R^d} (t-s)^{\alpha-1}G(t-s,x-y)f(s,y)\: \ud s \ud y \shortintertext{and}
  \left(Y_{\alpha} f\right)(t,x) & \coloneqq \int_0^{t} \int_{\R^d} (t-s)^{-\alpha}G(t-s,x-y)f(s,y) W(\ud s,\ud y) \:.
\end{align}

For $F_\alpha$, the first step of the proof of~\cite[Theorem 3.1]{tessitore.zabczyk:98:invariant}
showed the following proposition:
\begin{proposition}\label{P:F-Comp}
  Let $\rho$ and $\tilde{\rho}$ be given as in condition (i) of Theorem~\ref{T:Main}
  (see~\eqref{E:rhorho}). For any $q>2$, $t_0>0$ and $\alpha \in (q^{-1},2^{-1})$, the operator
  $F_\alpha$, as an operator from $L^q\big((0,t_0);\: L^2_{\tilde{\rho}}(\R^d)\: \big)$ to
  $L^2_\rho(\R^d)$, is compact.
\end{proposition}

As for $Y_\alpha$, we have the following two lemmas, which hold for both the non-restarted SHE
($t_0 = 0$) and the restarted SHE ($t_0 > 0$).

\begin{lemma}\label{L:yineq}
  Suppose that $\mu$---the initial condition for $u$---satisfies~\eqref{A:icon} and that
  $\widehat{f}$ satisfies Dalang's condition~\eqref{E:Dalang}. Suppose that for some $\alpha\in
  (0,1/2)$, $\mathcal{H}_\alpha(t)$ defined in~\eqref{E:H} is finite for all $t>0$. Fix an arbitrary
  $t_0\ge 0$. Let $v(t,x)$ be the solution to the restarted SHE~\eqref{E:restart} and
  $\dot{W}_{t_0}$ be the time-shifted noise (see~\eqref{E:ShiftW}) when $t_0>0$ and let $v = u$ when
  $t_0 = 0$. Then
  \begin{equation}\label{E:Y}
     Y_v(s,y) \coloneqq \left[Y_\alpha b\left(\circ, v(\cdot,\circ)\right)\right](s,y)
               = \int_0^s \int_{\R^d} (s-r)^{-\alpha}G(s-r,y-z)b(z,v(r,z))W_{t_0}(\ud r,\ud z)
  \end{equation}
  has the following properties:

  \begin{enumerate}[(1)]
    \item for all $q\ge 2$, $s>0$, $y\in\R^d$,
      \begin{align}\label{E:NormY}
        \Norm{Y_v(s,y)}_q^2 \le H\left(s+t_0;32qL_b^2\right) \: J_0^2 (s + t_0,y;\mu^*)\: \mathcal{H}_\alpha(s)<\infty,
      \end{align}
      where we remind the reader that $\mu^* \coloneqq 1 + |\mu|$, and we refer to
      Theorem~\ref{T:ExUn} for the function $H\left(t;\gamma\right)$;
    \item under both conditions in~\eqref{E:U0finite}, if $\mathcal{H}_\alpha(t)$ is finite for some
      $\alpha\in \left(64 L_b^2 \Upsilon(0),\: 1/2\right)$, then for any $q$ with $1/q\in \left(64
      L_b^2 \Upsilon(0), \: \alpha\right)$, the function $H\left(t;32qL_b^2\right)$
      in~\eqref{E:NormY} is uniformly bounded in $t\ge 0$, i.e., $\sup_{t\ge
      0}H\left(t;32qL_b^2\right)<\infty$;
    \item under both conditions in~\eqref{E:U0finite}, if $\mathcal{H}_\alpha(t)$ is finite for some
      $\alpha\in \left(64 L_b^2 \Upsilon(0),\: 1/2\right)$, then for any $q$ with $1/q\in \left(64
      L_b^2 \Upsilon(0), \: \alpha\right)$ and for any nonnegative and $L^1 (\R^d)$-function $\rho$,
      there exists a constant $\Theta = \Theta\left(q,L_b,L_0,\alpha\right)$, which does not depend
      on $t$, such that for $t>0$,
      \begin{equation}\label{E:yineq}
        \E \left(\int_0^{t} \Norm{Y_v(s,\cdot)}_{\rho}^q \ud s\right)
        \le \Theta \int_0^t \left[\mathcal{G}_{\rho}(s + t_0;\mu^*)\: \mathcal{H}_\alpha(s)\right]^{q/2} \ud s,
      \end{equation}
      which is finite provided that
    \begin{align}\label{E:GH-int'}
      \int_{0}^t \left[\mathcal{G}_{\rho}(s + t_0;\mu)\:\mathcal{H}_\alpha(s)\right]^{q/2} \ud s <\infty.
    \end{align}
  \end{enumerate}
\end{lemma}
\begin{remark}
  Condition~\eqref{E:GH-int'} is true for $t_0>0$ because $\mathcal{G}_\rho(t;\mu)$ is a continuous
  function for $t>0$ and $\mathcal{H}_\alpha(s)$ is continuous and bounded for $s\in [0,t]$ thanks
  to~\eqref{E:H}. However, when $t_0 = 0$, the situation is much trickier. For example, when the
  initial condition is the delta initial condition, we have that
  \begin{align*}
    \mathcal{G}_{\tilde{\rho}}(t;\delta_0)
    = \int_{\R^d} G(t,x)^2 \rho(x) \ud x
    = G(2t,0) \int_{\R^d} G(t/2,x)\rho(x)\ud x
    < \infty,
  \end{align*}
  where one can obtain the second equality via~\eqref{E:gsim}. Hence, when $s\to 0$,
  $\mathcal{G}_{\tilde{\rho}}(s;\delta_0)$ blows up with a rate $s^{-d/2}$. Considering that
  $\mathcal{H}_\alpha(s)$ goes to zero with a different rate, one needs to combine these two rates
  to check if condition~\eqref{E:GH-int'} holds. By introducing $t_0$ and restarting the heat
  equation, one can avoid this issue, that being the potential singularity of
  $\mathcal{G}_{\tilde{\rho}}$ at $s = 0$.
\end{remark}
\begin{proof}
  In the proof, we use $C$ to denote a generic constant that may change its value at each
  appearance. We first prove~\eqref{E:NormY}. By the Burkholder-Davis-Gundy inequality and
  Minkowski's integral inequality, we see that
  \begin{align*}
    \Norm{Y_v(s,y)}_q^2 \le
    C \int_0^s \ud r \: (s-r)^{-2\alpha}
    \iint_{\R^{2d}} \ud z_1\ud z_2 \: & G(s-r,y-z_1) \Norm{b\left(z_1,v(r,z_1)\right)}_q \\
    \times f(z_1-z_2)                 & G(s-r,y-z_2) \Norm{b\left(z_2,v(r,z_2)\right)}_q.
  \end{align*}
  Note that for the Lipschitz condition in~\eqref{E:lipcon}, we have that
  \begin{align*}
    \left|b\left(x,v\right)\right|
    & \le \left|b\left(x,v\right) - b\left(x,0\right)\right| + \left|b\left(x,0\right)\right|
    \le L_b |v| + L_0 \le C (1+|v|),\quad C\coloneqq L_b\vee L_0.
  \end{align*}
  We apply this and the moment bound~\eqref{E:mb} to $\Norm{b(z_i, v(r,z_i))}_q$ above to see that
  \begin{align}\label{E:NormB}
    \Norm{b(z_i, v(r,z_i))}_q
    & \le C\left(1+\Norm{v(r,z_i)}_q\right)                               \notag \\
    & =   C\left(1+\Norm{u(r+t_0,z_i)}_q\right)                           \notag \\
    & \le C H\left(r+t_0;32qL_b^2\right) J_0\left(r + t_0,z_i;\mu^*\right)\notag \\
    & \le C H\left(s+t_0;32qL_b^2\right) J_0\left(r + t_0,z_i;\mu^*\right), \quad i = 1,2, \: r\in(0,s),
  \end{align}
  where the last step is due to the fact that $H(t;\gamma)$ is a nondecreasing function; see Lemma~2.6 of~\cite{chen.kim:19:nonlinear}. Therefore, by denoting $C_s \coloneqq
  H\left(s+t_0;32qL_b^2\right)$,
  \begin{align*}
    \Norm{Y_v(s,y)}_q^2
    & \le C C_s \int_0^s \ud r \: (s-r)^{-2\alpha} \iint_{\R^{2d}} \ud z_1\ud z_2\: f(z_1-z_2) \prod_{i = 1}^2 \bigg(G(s-r,y-z_i) J_0\left(r + t_0,z_i;\mu^*\right)\bigg)      \\
    & = C C_s \int_0^s \ud r \: (s-r)^{-2\alpha} \iint_{\R^{2d}} \mu^* (\ud \sigma_1) \mu^* (\ud \sigma_2) \iint_{\R^{2d}} \ud z_1\ud z_2                                      \\
    & \quad \times f(z_1-z_2) \prod_{i = 1}^{2} \bigg(G(s-r,y-z_i)G(r + t_0,z_i-\sigma_i) \bigg)                                                                               \\
    & = C C_s\int_0^s \ud r \: (s-r)^{-2\alpha} \iint_{\R^{2d}} \mu^* (\ud \sigma_1) \mu^* (\ud \sigma_2)\: G(s + t_0,y-\sigma_1)G(s + t_0,y-\sigma_2)                         \\
    & \quad \times\iint_{\R^{2d}} \ud z_1\ud z_2 f(z_1-z_2) \prod_{i = 1}^2 G\left(\frac{(r+t_0)(s-r)}{s + t_0}, z_i-\sigma_i\frac{r + t_0}{s + t_0}-\frac{s-r}{s+t_0}y\right) \\
    & \le C C_s(2\pi)^{-2d} \int_0^s \ud r \: (s-r)^{-2\alpha} \iint_{\R^{2d}}\mu^* (\ud \sigma_1) \mu^* (\ud \sigma_2)\: G(s + t_0,y-\sigma_1)G(s + t_0, y-\sigma_2)          \\
    & \quad \times \int_{\R^d} \widehat{f}(\ud \xi) \exp\left(-\frac{(r + t_0)(s-r)}{s+ t_0}|\xi|^2\right),
  \end{align*}
  where we have applied~\eqref{E:gsim} and Plancherel's theorem. Hence,
  \begin{align*}
   \Norm{Y_v(s,y)}_q^2 \le C C_s (2\pi)^{-2d} J_0^2\left(s + t_0,y;\mu^*\right) \int_0^s \ud r \: (s-r)^{-2\alpha} \int_{\R^d} \widehat{f}(\ud \xi)\exp\left(-\frac{(r + t_0)(s-r)}{s + t_0}|\xi|^2\right).
  \end{align*}
  Because the function
  \begin{align*}
    t_0 \mapsto \frac{r + t_0}{s + t_0} = 1- \frac{s-r}{s+t_0} \quad \text{for $t_0 > 0$},
  \end{align*}
  is nondecreasing in $t_0 $ whenever $s > r > 0$, we can replace the two appearances of $t_0$ in
  the exponent of the above inequality by zero to see that
  \begin{equation}\label{E_:NormY}
    \hspace{-1em}
    \Norm{Y_v(s,y)}_q^2 \le C C_s (2\pi)^{-2d} J_0^2 (s + t_0,y;\mu^*) \int_0^s \ud r \: (s-r)^{-2\alpha} \int_{\R^d} \widehat{f}(\ud \xi)\exp\left(-\frac{r(s-r)}{s}|\xi|^2\right).
  \end{equation}
  Furthermore, by symmetry of $r(s-r)/s$ and the fact that $r(s-r)/s \ge r/2$ for all $r\in
  [0,s/2]$, we see that the above double integral is bounded by
  \begin{align*}
    \le & \: 2\int_0^{s/2} \ud r \: r^{-2\alpha}\int_{\R^d} \widehat{f}(\ud \xi)\exp\left(-\frac{r}{2}|\xi|^2\right)      \\
    =   & \: 2^{2(1-\alpha)} \int_0^{s/4} \ud r \: r^{-2\alpha}\int_{\R^d} \widehat{f}(\ud \xi)\exp\left(-r|\xi|^2\right) \\
    \le & \: 2^{2(1-\alpha)} \int_0^{s} \ud r \: r^{-2\alpha}\int_{\R^d} \widehat{f}(\ud \xi)\exp\left(-r|\xi|^2\right)   \\
    =   & \: 2^{2(1-\alpha)} \mathcal{H}_\alpha(s).
  \end{align*}
  Plugging the above bound back to~\eqref{E_:NormY} proves~\eqref{E:NormY}. \bigskip

  Part (2) is a direct consequence of Theorem~\ref{T:ExUn}. It remains to prove~\eqref{E:yineq}. An
  application of Minkowski's inequality shows that
  \begin{align}\label{E:EYqr}
    \E\left(\Norm{Y_v(s,\cdot)}_{\rho}^q\right)
    = \Norm{\int_{\R^d} Y_v(s,y)^2\rho(y) \ud y}_{q/2}^{q/2}
    \le \left(\int_{\R^d} \Norm{Y_v(s,y)}_q^2\rho(y) \ud y\right)^{q/2}.
  \end{align}
  By the definition of $\mathcal{G}_{\rho}(t;\mu)$ in~\eqref{E:G} and by~\eqref{E:NormY}, we see
  that
  \begin{align*}
    \int_{\R^d}\Norm{Y_v(s,y)}_q^2 \: \rho(y)\ud y
    & \le C \: \mathcal{G}_{\rho}(s + t_0;\mu^*) \mathcal{H}_\alpha(s).
  \end{align*}
  Plugging the above expression to the far right-hand side of~\eqref{E:EYqr} proves~\eqref{E:yineq}.
  Finally, the finiteness of the upper bound in~\eqref{E:yineq} is guaranteed by
  condition~\eqref{E:GH-int'}. This completes the proof of Lemma~\ref{L:yineq}.
\end{proof}

\begin{lemma}[Factorization lemma]\label{L:Factorize}
  Suppose that $\mu$ --- the initial condition for $u$ --- satisfies~\eqref{A:icon} and
  $\widehat{f}$ satisfies Dalang's condition~\eqref{E:Dalang}. Assume that condition~\eqref{E:H} is
  satisfied for some $\alpha\in (0,1/2)$. Fix an arbitrary $ t_0\ge 0$. Let $v(t,x)$ be the solution
  to the restarted SHE~\eqref{E:restart} and $\dot{W}_{t_0}$ be the time-shifted noise
  (see~\eqref{E:ShiftW}) when $t_0>0$ and let $v = u$ when $ t_0 = 0$. Then the following
  factorization holds
  \begin{equation*}
    \dfrac{\sin(\alpha \pi)}{\pi}
    \int_0^t (t-s)^{\alpha-1}\left[G(t-s,\cdot)*Y_v(s,\cdot)\right](x) \ud s
    = \int_0^t \int_{\R^d} G(t-r,x-z)b(z,v(r,z)) W_{t_0}(\ud r,\ud z),
  \end{equation*}
  for all $t>0$ and $x\in\R^d$. As a consequence,
  \begin{equation}\label{E:1fac}
    v(t,x) = \left[G(t, \cdot) * u(t_0,\cdot; \mu)\right](x) +\dfrac{\sin(\alpha \pi)}{\pi} \left[F_\alpha Y_v\right](t,x),
    \quad \text{for all $t>0$ and $x\in\R^d$}.
  \end{equation}
\end{lemma}
\begin{proof}
  The lemma is straightforward provided that one can switch the orders of stochastic and ordinary
  integrals:
  \begin{align}
    \MoveEqLeft[4] \int_0^t (t-s)^{\alpha-1}\left[G(t-s,\cdot)*Y_v(s,\cdot)\right](x) \ud s \notag \\
    = & \int_0^t \ud s\: (t-s)^{\alpha-1}\int_{\R^d} \ud y\: G(t-s,x-y) \notag                                                                \\
      & \times \int_0^s\: \int_{\R^d} \: (s-r)^{-\alpha}G(s-r,y-z)b(z,v(r,z))W_{t_0}(\ud r,\ud z) \notag                                      \\
    = & \int_0^t \ud s\: (t-s)^{\alpha-1} \int_0^s \int_{\R^d} \: (s-r)^{-\alpha}G(t-r,x-z)b(z,v(r,z))W_{t_0}(\ud r,\ud z) \label{E:Fubini_1} \\
    = & \int_0^t \int_{\R^d} W(\ud r,\ud z)G(t-r,x-z)b(z,v(r,z)) \int_r^t \ud s\: (s-r)^{-\alpha} (t-s)^{\alpha-1}\label{E:Fubini_2}          \\
    = & \dfrac{\pi}{\sin(\alpha \pi)}\int_0^t \int_{\R^d} G(t-r,x-z)b(z,v(r,z)) W_{t_0}(\ud r,\ud z),\notag
  \end{align}
  where the last step is the \textit{Beta integral} which requires that $\alpha\in (0,1)$. It
  remains to justify the two applications of the stochastic Fubini's theorem (see Theorem 5.30 of
  Chapter one in~\cite{dalang.khoshnevisan.ea:09:minicourse}, or also~\cite{walsh:86:introduction}
  or Theorem 4.33 of~\cite{da-prato.zabczyk:14:stochastic}) in~\eqref{E:Fubini_1}
  and~\eqref{E:Fubini_2} in the following two steps. \bigskip

  \textbf{\noindent Step 1.~} In this step, we justify the change of orders in~\eqref{E:Fubini_1}.
  Note that $t, x$ and $s$ are fixed. It suffices to prove the following condition:
  \begin{align*}
    I_1\coloneqq & \int_{\R^d}\ud y \: G(t-s,x-y) \int_0^s \ud r\: (s-r)^{-2\alpha} \iint_{\R^{2d}} \ud z_1\ud z_2\:                \\
                 & \times f(z_1-z_2) \left(\prod_{i = 1}^{2} G(s-r,y-z_i)\right) \: \E\left(\prod_{i = 1}^{2}b(z_i,v(r,z_i))\right) \\
       =         & \int_{\R^d}\ud y \: G(t-s,x-y) \Norm{Y_v(s,y)}_2^2 < +\infty,
  \end{align*}
  which follows immediately from~\eqref{E:NormY}. Indeed,
  \begin{align*}
    \int_{\R^d}\ud y \: G(t-s,x-y) \Norm{Y_v(s,y)}_2^2
    \le & C \int_{\R^d} \ud y \; G(t-s,x-y) J_0^2 (s+t_0,y;\mu^*) \mathcal{H}_\alpha(s)                        \\
     =  & C \mathcal{H}_\alpha(s) \int_{\R^d} \ud y \; G(t-s,x-y)\iint_{\R^{2d}}\mu^* (\ud z_1)\mu^* (\ud z_2) \\
        & \times G\left(s+t_0,y-z_1\right)G\left(s+t_0,y-z_2\right).
  \end{align*}
  Now we bound the three heat kernels using~\eqref{E:gsim} as follows:
  \begin{align*}
     \MoveEqLeft G (t-s,x-y) \prod_{i = 1}^{2}G(s+t_0,y- z_i)\\
       & = \frac{G\left(2(t-s),x-y\right)^2}{G\left(4(t-s,0\right)} \prod_{i = 1}^{2}G\left(s+t_0,y- z_i\right)                                                  \\
       & \le 2^d \frac{G\left(2(t-s),x-y\right)^2}{G\left(4(t-s),0\right)} \prod_{i = 1}^{2}G\left(2s+2t_0,y- z_i\right)                                         \\
       & = 2^d [4(t-s)]^{d/2} \prod_{i = 1}^{2}\bigg[ G\left(2s+2t_0,y- z_i\right)G\left(2(t-s),x-y\right)\bigg]                                                 \\
       & = 2^{2d} (t-s)^{d/2} \prod_{i = 1}^{2}\left[G\left(2(t+t_0),x- z_i\right)G\left(\frac{2(t-s)(s+t_0)}{t+t_0},y-\frac{s+t_0}{t+t_0}(x- z_i)\right)\right] \\
       & \le 2^{2d} (t-s)^{d/2} \prod_{i = 1}^{2}\left[G\left(2(t+t_0),x- z_i\right)G\left(\frac{2(t-s)(s+t_0)}{t+t_0},0\right)\right]                           \\
       & \le C_{t,s,t_0} \prod_{i = 1}^{2} G\left(2(t+t_0),x- z_i\right).
  \end{align*}
  Therefore, $I_1\le C_{t,s,t_0} \: \mathcal{H}_{\alpha}(s) \: J_0^2\left(2(t+t_0),x;\mu^*\right)
  <\infty$. \bigskip

  \textbf{\noindent Step 2.~} Similarly, as for~\eqref{E:Fubini_2}, we need to show that
  \begin{align*}
    I_2\coloneqq & \int_0^t \ud s \: (t-s)^{\alpha-1} \int_0^s \ud r\: (s-r)^{-2\alpha} \iint_{\R^{2d}} \ud z_1\ud z_2\: \\
                 & \times f(z_1-z_2) \left(\prod_{i = 1}^{2} G(t-r,x-z_i)\right) \: \E\left(\prod_{i = 1}^{2}b(z_i,v(r,z_i))\right) < \infty.
  \end{align*}
  By the Cauchy Schwartz inequality,~\eqref{E:NormB} and because $\alpha\in (0,1/2)$,
  \begin{align*}
    I_2\le & C \int_0^t \ud s \: (t-s)^{\alpha-1} \int_0^s \ud r\: (s-r)^{-2\alpha}                                                        \\
           & \times \iint_{\R^{2d}} \ud z_1\ud z_2\:f(z_1-z_2) \prod_{i = 1}^{2}\bigg(G(t-r,x-z_i) J_0\left(r + t_0,z_i;\mu^*\right)\bigg) \\
         = & C' \int_0^t \ud r \: (t-r)^{-\alpha} \iint_{\R^{2d}} \ud z_1\ud z_2\:f(z_1-z_2) \prod_{i = 1}^{2}\bigg(G(t-r,x-z_i) J_0\left(r + t_0, z_i;\mu^*\right)\bigg).
  \end{align*}
  Now by the same arguments as those leading to~\eqref{E:NormY} (with $2\alpha$ there replaced by
  $\alpha$), we see that
  \begin{align*}
    I_2 & \le C J_0^2\left(t,x; \mu^*\right) \int_0^t \ud r \: r^{-\alpha} \int_{\R^d} \widehat{f}(\ud \xi)\exp\left(-\frac{r(t-r)}{t}|\xi|^2\right),
  \end{align*}
  which is finite by~\eqref{E:H} where we replace $\alpha$ with $\alpha/2$ and repeat the same steps
  right after~\eqref{E_:NormY}. This completes the proof of Lemma~\ref{L:Factorize}.
\end{proof}

\begin{codes}

Integrate[(s-r)^{-a} (t-s)^{a-1},{s,r,t},Assumptions->{t>s>0}]

Integrate[(t-s)^{a-1} (s-r)^{-2a},{s,r,t},Assumptions->{t>r>0,0<a<1/2}]

\end{codes}

\begin{codes}

  Integrate[Exp[-s x],{s,0,Infinity}, Assumptions->x>0]

\end{codes}

Finally, we characterize conditions~\eqref{E:Dalang-a0} and~\eqref{E:H} in the following lemma:

\begin{lemma}\label{L:HUA}
  For all $\alpha\in (0,1/2]$, we have the following properties:
  \begin{enumerate}[(1)]
  \item $\left(2\pi\right)^{-d}\mathcal{H}_\alpha(t)\le
    \Gamma\left(1-2\alpha\right)\Upsilon_{2\alpha}(0)$ for all $t>0$ and hence
    \begin{align}\label{E:HUA}
      \Upsilon_{2\alpha}(0)<\infty \quad \Longrightarrow \quad \mathcal{H}_\alpha(t)<\infty \quad \text{for
      all $t>0$};
    \end{align}
  \item $\lim_{t\to\infty}\left(2\pi\right)^{-d}\mathcal{H}_\alpha(t)= \Gamma\left(1-2\alpha\right)
    \Upsilon_{2\alpha}(0)$;
  \item if $\Upsilon(0)<\infty$, then the reverse implication of~\eqref{E:HUA} holds.
\end{enumerate}
\end{lemma}
\begin{proof}
  We only need to consider the case when $\alpha>0$. It is clear that the function
  $\mathcal{H}_\alpha(t)$ is nondecreasing. Hence, part (2) implies part (1). As for part (2), by
  Fubini's theorem,
  \begin{align}\label{E:Fubini}
    \lim_{t\to\infty}\mathcal{H}_\alpha(t)
    & = \int_0^\infty \ud r \: r^{-2\alpha} \int_{\R^d} \widehat{f}(\ud\xi)e^{-r|\xi|^2}
      = \Gamma(1-2\alpha) \int_{\R^d} \frac{\widehat{f}(\ud\xi)}{|\xi|^{2\left(1-2\alpha\right)}}
      = C \Upsilon_{2\alpha}(0),
  \end{align}
  with $C\coloneqq \Gamma(1-2\alpha)(2\pi)^d$. Now for part (3), for any $t>0$, by splitting the $\ud r$
  integral in~\eqref{E:Fubini} into two parts, we see that
  \begin{align*}
    C \Upsilon_{2\alpha}(0) & = \int_0^\infty \ud r \: r^{-2\alpha} \int_{\R^d} \widehat{f}(\ud\xi)\exp\left(-r|\xi|^2\right) = \mathcal{H}_\alpha(t) + I_\alpha(t), \quad \text{with} \\
    I_\alpha(t)             & = \int_t^\infty \ud r \: r^{-2\alpha} \int_{\R^d}
    \widehat{f}(\ud\xi)\exp\left(-r|\xi|^2\right).
  \end{align*}
  Notice that
  \begin{align*}
    I_\alpha(t) \le t^{-2\alpha} \int_t^\infty \ud r \int_{\R^d} \widehat{f}(\ud\xi)e^{-r|\xi|^2}
    = t^{-2\alpha} \int_{\R^d} \frac{\widehat{f}(\ud\xi)}{|\xi|^2} e^{-t|\xi|^2}
    \le t^{-2\alpha} \int_{\R^d} \frac{\widehat{f}(\ud\xi)}{|\xi|^2}
    = \frac{\left(2\pi\right)^d}{t^{2\alpha}} \Upsilon(0).
  \end{align*}
  Therefore,
  \begin{align*}
    \Upsilon_{2\alpha}(0) \le \frac{\mathcal{H}_\alpha(t)}{(2\pi)^d\Gamma(1-2\alpha)} +
    \frac{\Upsilon(0)}{\Gamma\left(1-2\alpha\right)t^{2\alpha}} <\infty, \quad \text{for all $t>0$,}
  \end{align*}
  which proves part (3).
\end{proof}

\section{Tightness and construction -- Proof of Theorem~\ref{T:Main}}\label{S:mainproof}
\subsection{Proof of part (a) of Theorem~\ref{T:Main}}

Before we start the proof of part (a) of Theorem~\ref{T:Main}, we first recall the following result:

\begin{proposition}[Proposition 2.1 of~\cite{tessitore.zabczyk:98:invariant}]\label{P:weight}
  For any admissible weight $\rho$, the operators on $L^2_\rho(\R^d)$ defined by $\varphi \mapsto
  \big(G(t,\cdot)*\varphi(\cdot)\big)(x)$ can be extended to a $C_0-semigroup$ on $L^2_\rho(\R^d)$.
  Moreover, if $\tilde{\rho}$ is another admissible weight such that
	\begin{align*}
	\int_{\R^d}\dfrac{\rho(x)}{\tilde{\rho}(x)}\ud x < \infty,
	\end{align*}
	then for any $t>0$, the operators defined above are compact from $L^2_{\tilde{\rho}}(\R^d)$ to
	$L^2_\rho(\R^d)$.
\end{proposition}

\begin{proof}[Proof of Theorem~\ref{T:Main} (a)]
  In this proof, $u(t,x)$ refers to $u(t,x;\mu)$. Fix $\tau>0$ and let $t_0 = \tau/2$. Throughout the
  proof, we have $t\ge \tau$. Let $v$ be the solution to~\eqref{E:restart} that is restarted from
  $t- t_0$. Then (see Figure~\ref{F:RestartedSHE} for an illustration)

  \begin{align}\label{E:RestaredSHE}
    v_t\left(s,x\right) \stackrel{\mathcal{L}}{=} u\left(s,x; u\left(t-t_0,\cdot;\mu\right)\right) \quad \text{for $s \ge 0$ and $t\ge\tau$}.
  \end{align}

\begin{figure}[htpb]
  \begin{center}
    \begin{tikzpicture}[scale = 1, transform shape]
      \tikzset{> = latex}
      \draw [->] (-0.2,0) --++ (8.5,0) node[right] {$s$};
      \draw[thick] (0,0) --++(0,0.2) --++(0,-0.4) node[below] {$0$};
      \draw (1,0) --++(0,0.2) --++(0,-0.4) node[below] {$t_0$};
      \draw (2,0) --++(0,0.2) --++(0,-0.4) node[below] {$\tau$};
      \draw (4,0) --++(0,0.2) --++(0,-0.4) node[below] {$t-t_0$};
      \draw (5,0) --++(0,0.2) --++(0,-0.4) node[below] {$t$};
      \draw (6.5,0) --++(0,0.2) --++(0,-0.4) node[below] {$t-t_0+s$}
            ++ (0,-0.7) node[below] (u1) {$u(t-t_0+s,x;\mu)$}
            ++ (0,-0.6) node[below] {$||$}
            ++ (0,-0.6) node[below] {$u\left(s,x;u(t-t_0,\cdot;\mu)\right)$};

      \draw [->] (3.8,1) --++ (4.5,0) node[right] {$s$};
      \draw[thick] (4,1) --++(0,-0.2) --++(0,0.4) node[above] {$0$};
      \draw (5,1) --++(0,-0.2) --++(0,0.4) node[above] {$t_0$};
      \draw (6.5,1) --++(0,-0.2) --++(0,0.4) node[above] {$s$} ++ (0,0.5) node[above] (v) {$v_t(s,x)$};

      \draw[dotted] (6.5,0.2) -- (6.5,0.8);
      \draw[dotted] (5,0.2) -- (5,0.8);
      \draw[dotted] (4,0.2) -- (4,0.8);
    \end{tikzpicture}
  \end{center}
  \caption{An illustration for the restarted SHE in~\eqref{E:RestaredSHE}.}
  \label{F:RestartedSHE}
\end{figure}
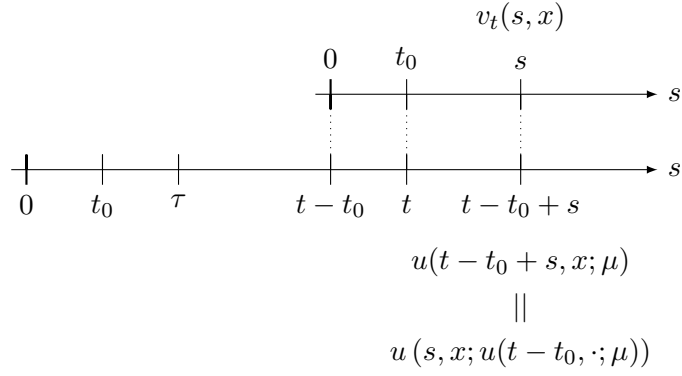

  According to Assumption (i), we can choose and fix some admissible weight function $\tilde{\rho}$
  such that~\eqref{E:rhorho} is satisfied. Hence, by Proposition~\ref{P:weight} below, the following
  set
  \begin{align*}
    \mathscr{K}_1 (\Lambda) \coloneqq \left\{\big(G(t_0,\cdot)*y(\cdot)\big)(x): \:\:\Norm{y}_{\tilde{\rho}} \le \Lambda \right\}
    \quad \text{with $\Lambda>0$}
  \end{align*}
  is relatively compact in $L_\rho^2 (\R^d)$.

  Assumption (iii), i.e.,~\eqref{E:U0finite}, implies that the interval $\left(64L_b^2 \Upsilon(0),
  1/2\right)$ is not empty. Moreover, Assumption (iv), i.e.,~\eqref{E:Dalang-a0}, guarantees that
  there exists a constant $\alpha$ in this interval, namely, $64L_b^2\Upsilon(0) <\alpha<1/2$, such
  that~\eqref{E:Dalang-a0} holds with $\alpha$ replaced by $2\alpha$, i.e.,
  $\Upsilon_{2\alpha}(0)<\infty$. Now we can apply part (3) of Lemma~\ref{L:HUA}, thanks
  to~\eqref{E:Dalang-00}, to see that $\Upsilon_{2\alpha}(0)<\infty$ if and only if~\eqref{E:H}
  holds. Therefore, both Lemmas~\ref{L:yineq} and~\ref{L:Factorize} (more precisely part (3) of
  Lemma~\ref{L:yineq}) are applicable. In particular, Lemma~\ref{L:Factorize} ensures that the
  following factorization is well-defined:
  \begin{align}\label{E:Fact}
    v(t_0,x) = \Big(G(t_0,\cdot) * u(t-t_0,\cdot)\Big)(x) +\dfrac{\sin(\alpha \pi)}{\pi} \left[F_\alpha Y_v\right](t_0,x).
  \end{align}
  Part (3) of Lemma~\ref{L:yineq} shows that for any $q$ in the following range,
  \begin{align}\label{E:Choice}
    64 L_b^2 \Upsilon(0) < \frac{1}{q} < \alpha <\frac{1}{2}
    \qquad \left(\text{or equivalently} \quad
    2< \frac{1}{\alpha} < q < \frac{1}{64 L_b^2 \Upsilon(0)}\:\right),
  \end{align}
  we can apply Proposition~\ref{P:F-Comp} to see that the set
  \begin{align*}
    \mathscr{K}_2 (\Lambda)\coloneqq
    \left\{(F_\alpha h)(t_0,x): \Norm{h}_{L^q\left((0,t_0);L^2_{\tilde{\rho}}\left(\R^d\right)\right)} \le \Lambda \right\},
    \quad \text{with $\Lambda>0$},
  \end{align*}
  is relatively compact in $L_\rho^2 (\R^d)$. Now for any $\Lambda>0$, define the set $\mathscr{K}
  (\Lambda)$ as
  \begin{align*}
    \mathscr{K} (\Lambda)
    & \coloneqq \mathscr{K}_1 (\Lambda) + \mathscr{K}_2 (\Lambda) \\
    & = \left\{\big(G(t_0,\cdot)*y(\cdot)\big)(x)+(F_\alpha h)(t_0, x): \:\Norm{y}_{\tilde{\rho}}
    \le \Lambda \quad \text{and}\quad\Norm{h}_{L^q\left((0,t_0);L^2_{\tilde{\rho}}\left(\R^d\right)\right)} \le \Lambda \right\}.
  \end{align*}

  Notice that from the factorization formula~\eqref{E:Fact},
  \begin{align*}
    \mathbb{P} \left[ v(t_0,\cdot) \not \in \mathscr{K}(\Lambda) \right ]
    & \le \mathbb{P}\left[ \left(\int_{0}^{t_0} \Norm{Y_v(s,\cdot)}_{\tilde{\rho}}^q \: \ud s\right)^{1/q} > \dfrac{\pi \Lambda}{\sin(\alpha \pi )} \right]
    + \mathbb{P}\left[ \Norm{u(t-t_0,\cdot)}_{\tilde{\rho}} > \Lambda \right] \\
    & \eqqcolon I_1+I_2.
  \end{align*}
  By Chebyshev's inequality and~\eqref{E:uInLrho}, we see that
  \begin{align*}
    I_2\le\frac{1}{\Lambda^2} \mathbb{E}\left(\Norm{u(t-t_0,\cdot)}_{\tilde{\rho}}^2\right)
    \le \frac{1}{\Lambda^2} \mathcal{G}_{\tilde{\rho}}\left(t-t_0;\mu^*\right).
  \end{align*}
  Because $\mathcal{G}_{\tilde{\rho}}(t;\mu^*)$ is a continuous function for $t>0$, and because it
  is also bounded at infinity, thanks to Assumption (ii) (see~\eqref{E:InitData}), we have that
  \begin{align}\label{E:GtBdd}
    \mathcal{G}_{\tilde{\rho}}\left(t-t_0;\mu^*\right)
    \le \sup_{t\ge \tau} \mathcal{G}_{\tilde{\rho}}\left(t-t_0;\mu^*\right)
    = \sup_{t\ge t_0} \mathcal{G}_{\tilde{\rho}}\left(t;\mu^*\right)<\infty.
  \end{align}
  Therefore, we can bound $I_2$ from above with a constant that does not depend on $t\ge \tau$,
  namely,
  \begin{align*}
    I_2 \le\frac{1}{\Lambda^2} \sup_{t\ge t_0} \mathcal{G}_{\tilde{\rho}}\left(t;\mu^*\right)<\infty.
  \end{align*}

  As for $I_1$, with the choice of $\alpha$ and $q$ in~\eqref{E:Choice}, one can apply Chebyshev's
  inequality and part (3) of Lemma~\ref{L:yineq} to see that
  \begin{align*}
    I_1 & \le \dfrac{\sin^q (\alpha \pi)}{\pi^q \Lambda^q}\E \int_{0}^{t_0} \Norm{Y_v(s,\cdot)}_{\tilde{\rho}}^q \: \ud s
          \le \dfrac{\sin^q (\alpha \pi)}{\pi^q \Lambda^q} \Theta\int_0^{t_0}
          \left(\mathcal{G}_{\tilde{\rho}}(s + t-t_0;\mu^*) \mathcal{H}_{\alpha}(s)\right)^{q/2} \ud s,
  \end{align*}
  where the constant $\Theta$ does not depend on $t$. As we have seen from above, since
  $\Upsilon_{2\alpha}(0)<\infty$, we can apply Lemma~\ref{L:HUA} to bound $\mathcal{H}_{\alpha}(s)$
  from above by the following finite bound: $\left(2\pi\right)^d \Gamma(1-2\alpha) \Upsilon_{2\alpha}(0)$.
  Hence, together with~\eqref{E:GtBdd}, we obtain the following upper bound for $I_1$ that is
  uniform in $t\ge \tau$:
  \begin{align*}
    I_1 \le \dfrac{\sin^q (\alpha \pi)\Theta (2\pi)^{dq/2} t_0}{\Gamma\left(1-2\alpha\right)^{q/2}\pi^q \Lambda^q} \left(\sup_{t\ge t_0} \mathcal{G}_{\tilde{\rho}}\left(t;\mu^*\right)\right)^{q/2} \Upsilon_{2\alpha}^{q/2}(0).
  \end{align*}
  Combining these two upper bounds, we see that
  \begin{align*}
    \mathbb{P} \left[ v(t_0,\cdot) \not \in \mathscr{K}(\Lambda) \right ]
    \le
    \dfrac{\sin^q (\alpha \pi)\Theta (2\pi)^{dq/2} t_0}{\Gamma\left(1-2\alpha\right)^{q/2}\pi^q \Lambda^q} \left(\sup_{t\ge t_0} \mathcal{G}_{\tilde{\rho}}\left(t;\mu^*\right)\right)^{q/2} \Upsilon_{2\alpha}^{q/2}(0)
    +
    \frac{1}{\Lambda^2} \sup_{t\ge t_0} \mathcal{G}_{\tilde{\rho}}\left(t;\mu^*\right)
    <\infty,
  \end{align*}
  with the upper bound holding uniformly for all $t\ge \tau$. Hence, for any $\epsilon>0$, by
  choosing $\Lambda > 0$ big enough such that
  \begin{align*}
    \dfrac{\sin^q (\alpha \pi)\Theta (2\pi)^{dq/2} t_0}{\Gamma\left(1-2\alpha\right)^{q/2}\pi^q \Lambda^q} \left(\sup_{t\ge t_0} \mathcal{G}_{\tilde{\rho}}\left(t;\mu^*\right)\right)^{q/2} \Upsilon_{2\alpha}^{q/2}(0)
    +
    \frac{1}{\Lambda^2} \sup_{t\ge t_0} \mathcal{G}_{\tilde{\rho}}\left(t;\mu^*\right)
    < \epsilon,
  \end{align*}
  we can ensure that
  \begin{align*}
      \mathbb{P}\left(u(t,\cdot) \in \mathscr{K}(\Lambda)\right)
    = \mathbb{P}\left(v(t_0,\cdot) \in \mathscr{K}(\Lambda)\right)
    \ge 1- \epsilon, \quad \text{for all $t\ge \tau$,}
  \end{align*}
  which proves part (a) of Theorem~\ref{T:Main}.
\end{proof}

\subsection{Proof of part (b) of Theorem~\ref{T:Main}}
\begin{proof}
  Fix an arbitrary $\tau>0$ and denote
  \begin{align*}
    U(T) \coloneqq \dfrac{1}{T}\int_{\tau}^{T+\tau} \mathcal{L}\left(u(t,\cdot;\mu)\right)\: \ud t, \quad T>0.
  \end{align*}
  We claim that the family of laws $U(T,\cdot)$ for $T > 0$ is tight in $L^2_{\rho}(\R^d)$. Indeed,
  for any $\epsilon\in (0,1)$, by part (a), there exists a compact set $\mathcal{K}\in L_\rho^2
  (\R^d)$ such that~\eqref{E:cpt} holds. This implies that
  \begin{align*}
    U(T)\left(\mathcal{K}\right)
    & = \frac{1}{T} \int_{\tau}^{T + \tau} \mathscr{L}(u(t,\cdot;\mu))\left(\mathcal{K}\right) \; \ud t
      \ge \frac{1}{T} \int_{\tau}^{T + \tau} (1-\epsilon) \; \ud t
      = 1-\epsilon, \quad \text{for all $T>0$}.
  \end{align*}

  Let $\{T_n\}_{n\in \mathbb{N}}$ be any deterministic sequence such that $T_n\uparrow \infty$.
  Since $\left\{U(T_n)\right\}_{n\ge 1}$ is a tight sequence of measures, then there exists a
  subsequence $\{U\left(T_{n_m}\right)\}_{m\ge 1}$ that converges weakly to a measure, $\eta$, on
  $L_\rho^2 (\R^d)$ (e.g. see~\cite[Theorem 5.1]{billingsley:99:convergence}). Then one can apply
  the Krylov-Bogoliubov existence theorem (see, e.g.,~\cite[Theorem
  11.7]{da-prato.zabczyk:14:stochastic}) to conclude that the measure $\eta$ is an invariant measure
  for $\mathscr{L}(u(t,\cdot;\mu))$, $t\ge \tau$. Finally, since $\tau$ can be arbitrarily close to
  zero, one can conclude part (b) of Theorem~\ref{T:Main}.
\end{proof}

\section{Discussion and examples}\label{S:Example}
\subsection{Invariant measures for SHE with a drift term}\label{SS:Drift}

In this part, we give a brief account of the case when the SHE has a drift term which plays a
crucial role in controlling the moments. The equation usually takes the following form:
\begin{align}\label{E:Drift}
  \left(\dfrac{\partial}{\partial t}-\dfrac{1}{2}\Delta\right) u(t,x)
   = g(x,u(t,x)) + b(x,u(t,x)) \dot W(t,x) \quad \text{$x\in \mathcal{O}$, $t>0$}.
\end{align}
The references in this part are far from being complete. The interested readers can find more
references from the references below.

The first case is when the drift term $g(\cdot)$ in~\eqref{E:Drift} satisfies certain dissipativity
conditions, which push the solution toward zero; see, e.g.,~\cite{assing.manthey:03:invariant,
brzezniak.gatarek:99:martingale, cerrai:01:second, cerrai:03:stochastic,
eckmann.hairer:01:invariant}. Such a ``negative'' drift term helps to cancel the growth of the
moments. Here is one example of such drift term: for some $m,\: k_i, \: c_i>0$ as $|u| \to \infty$:
\begin{equation}\label{E:dis}
  \begin{cases}
    g(u) \le -k_1|u|^m + k_2 & u > 0, \\
    g(u) \ge c_1|u|^m - c_2  & u < 0.
  \end{cases}
\end{equation}
In particular, Cerrai~\cite{cerrai:01:second, cerrai:03:stochastic} and Brze\'{z}niak and
G\c{a}tarek~\cite{brzezniak.gatarek:99:martingale} considered the case of a bounded spatial domain,
while Assing and Manthey~\cite{assing.manthey:03:invariant} and Eckmann and
Hairer~\cite{eckmann.hairer:01:invariant} considered the whole space $\R^d$. Note that Eckmann and
Hairer~\cite{eckmann.hairer:01:invariant} studied the additive noise case along with a bounded
initial condition.

Several works which do not require an added drift term with dissipativity as in \eqref{E:dis}
include Misiats \textit{et
al}~\cite{misiats.stanzhytskyi.ea:16:existence,misiats.stanzhytskyi.ea:20:invariant}. In the Theorem
1.2 of~\cite{misiats.stanzhytskyi.ea:20:invariant}, they provide a result guaranteeing the existence
of an invariant measure for the stochastic heat equation on the whole space $\R^d$. More precisely,
they allow for a drift term, $g(x,u):\R^d \times \R \to \R$, such that for all $x \in \R^d$ and
$u_1$, $u_2 \in \R$,
\begin{align*}
		|g(x,0)| \le \varphi(x) \qquad \text{and} \qquad |g(x,u_1) - g(x,u_2)| \le L\varphi(x)|u_1 - u_2|,
\end{align*}
for some $L>0$ where $\varphi(x)$ must decay fast enough such that $\varphi / \sqrt{\rho} \in
L^2 (\R^d) \cap L^{\infty}(\R^d)$ and $\rho$ is the admissible weight. Thus, $f \equiv 0$ is allowed.
However, they require the following condition on the diffusion term $b$:
\begin{align*}
  |b(x,u_1) - b(x,u_2)| \le L\varphi(x)|u_1 - u_2|,
\end{align*}
which excludes the parabolic Anderson model. Lastly, Theorem 1.2 \textit{ibid.} requires the initial
condition to be in $L^2 (\R^d)$, which excludes the two important cases, $u(0,x) = 1$ and
$u(0,\cdot) = \delta_0(\cdot)$. Our Theorem~\ref{T:Main} includes both of these initial conditions;
see Section~\ref{SS:Init}.

\subsection{The conditions for the spectral measures by Tessitore and Zabczyk}\label{SS:TessZabz}

Tessitore and Zabczyk~\cite{tessitore.zabczyk:98:invariant} established the existence of an
invariant measure for~\eqref{E:SHE} in $L^2_{\rho}(\R^d)$ under the assumptions that (1) there
exists a $\varphi \in L^2_{\rho}(\R^d) \cap L^2_{\widetilde\rho}(\R^d)$ where $\rho/\widetilde\rho
\in L^1 (\R^d)$ and the solution starting from $\varphi$ is bounded in probability in
$L^2_{\widetilde \rho}(\R^d)$ and (2) that the spectral density $\hat{f}$ satisfies
\begin{equation}\label{E:CorTZ}
\widehat{f} \in L^p (\R^d) \quad \text{where $\frac{d-2}{d} < \frac{1}{p}$};
\end{equation}
see Hypothesis 2.1 \textit{ibid.} However, as was illustrated in Theorem 3.3 \textit{ibid.}, in
order to apply this theorem to a specific initial condition in $L^2_{\rho}$ (or to have moments
uniformly bounded in time), the following additional assumptions were imposed:
\begin{equation}\label{E:TZspec}
  d\ge 3 \quad \text{and} \quad
  L_b^{-2}> \frac{\Gamma(d/2 -1)2^{d/2-2}}{(2\pi)^{2d}} \int_{\R^d} \left(\left| \mathcal{F}\big(\sqrt{\widehat{f}\:}\:\:\big) \right| * \left| \mathcal{F}\big(\sqrt{\widehat{f}\:}\:\:\big)\right|\right) (\xi) |\xi|^{2-d} \ud \xi ,
\end{equation}
where the convention of the Fourier transform is given in Remark~\ref{R:Fourier}. With these
assumptions, they were able to prove that~\eqref{E:SHE} starting from the constant one initial
condition satisfies~\eqref{E:BddMnt} and thus is bounded in probability, verifying the existence of
an invariant measure via the construction~\eqref{E:InvMeasForm}.


\begin{remark}\label{R:Fourier}
  The Fourier transform may be defined differently depending on how one handles the $2\pi$ constant.
  In this paper (as in~\cite{chen.kim:19:nonlinear,chen.huang:19:comparison}), we use the convention
  given in~\eqref{E:Fourier}. Hence, Plancherel's theorem takes the form of $\int_{\R^d} \psi(x)
  \phi(x)\ud x = (2\pi)^{-d} \int_{\R^d} \widehat{\psi}(\xi) \overline{\widehat{\phi}(\xi)}\ud \xi$.
  The authors in~\cite{tessitore.zabczyk:98:invariant} did not explicitly mention their convention
  of the Fourier transform. However, the proof of Theorem 3.3 \textit{ibid.} suggests that the
  following convention has been used:
	\begin{align*}
	\widehat{\phi}(\xi)      = \mathcal{F}\phi(\xi) := (2\pi)^{-d/2} \int_{\R^d} e^{-i x\cdot\xi} \phi(x)\ud x \quad \text{and} \quad
	\mathcal{F}^{-1}\psi(x) := (2\pi)^{-d/2} \int_{\R^d} e^{i x\cdot\xi} \psi(\xi)\ud \xi.
	\end{align*}
  Hence, Plancherel's theorem takes the form, $\int_{\R^d} \psi(x) \phi(x)\ud x = \int_{\R^d}
  \widehat{\psi}(\xi) \overline{\widehat{\phi}(\xi)}\ud \xi$, without the additional factor
  $(2\pi)^{-d}$. In particular, the spectral density $\gamma$ \textit{ibid.} corresponds to
  $\left(2\pi\right)^{-d/2}\widehat{f}$ in this paper. Our equation~\eqref{E:TZspec}, which is
  condition (3.4) \textit{ibid.}, takes into account this difference, therefore explaining the
  slightly different factor in front of the integral in~\eqref{E:TZspec} from that in (3.4)
  \textit{ibid.}
\end{remark}


In the following, we will focus on the second condition in~\eqref{E:TZspec}, which corresponds
to~\eqref{E:U0finite}. We claim that the latter is much easier to check than that of the former. The
square root and absolute value in~\eqref{E:TZspec} make their condition more restrictive. Indeed, if
$\mathcal{F}\big(\sqrt{\widehat{f}\:}\:\big)$ is nonnegative, then the absolute values
in~\eqref{E:TZspec} can be removed without ambiguity, which will reduce to our
condition~\eqref{E:U0finite} up to a constant factor. However, when $\widehat{f}$ is only
nonnegative and not strictly positive, finding the right square root of $\widehat{f}$ so that
$\mathcal{F}\big(\sqrt{\widehat{f}\:}\:\big)$ is nonnegative becomes tricky. This last point will be
illustrated by the examples below. We first set up the notation in Example~\ref{Ex:nonnegative}.
%
%


\begin{example}\label{Ex:nonnegative}
  Let $d = 1$ and $g(x) = \frac{1}{2}1_{[-1,1]}(x)$. Then we have that $ \widehat{g}(\xi) =
  \xi^{-1}\sin(\xi)$. Now set $f(x) = \left(g*g\right)(x) = 2^{-2}\max(2-|x|,0)$. It is clear that
  $f$ is nonnegative. It is also nonnegative-definite because $\widehat{f}(\xi) = \widehat{g}(\xi)^2
  = \xi^{-2}\sin^2 (\xi)\ge 0$.
\end{example}



Up to a constant, one may replace $\mathcal{F}$ by $\mathcal{F}^{-1}$. The following example shows
that $\mathcal{F}^{-1}\big(\sqrt{\widehat{f\:}}\:\:\big)(x)$ is signed and hence, the absolute value
make spoil the oscillatory structure.

\begin{example}\label{Ex:Finv}
  Suppose that $d = 1$, $f(x) = (2\pi)^{-1}x^{-2}\sin^2 (x)$ and $\widehat{f}(\xi) =
  2^{-2}\max\{2-|\xi|,0\}$. Then as explained in Example~\ref{Ex:nonnegative}, both $f$ and
  $\widehat{f}$ are non-negative and non-negative definite. We claim that
	\begin{align}\label{E:Claim}
	\text{$\mathcal{F}^{-1}\big(\sqrt{\widehat{f}\:}\:\:\big)(x)$ takes both positive and negative values.}
	\end{align}
	Indeed, for all $x\in\R$,
	\begin{align*}
    \mathcal{F}^{-1}\big(\sqrt{\widehat{f}\:}\:\:\big)(x)
    & = \frac{2}{2\pi} \int_0^2 2^{-1}\sqrt{2-\xi} \cos(x \xi) \: \ud \xi
      = \frac{1}{\pi} \int_0^{\sqrt{2}} \xi^2 \cos\left(x\left(2-\xi^2\right)\right) \: \ud \xi \\
    & = \frac{1}{\pi}\cos(2x) \int_0^{\sqrt{2}} \xi^2 \cos(x \xi^2) \: \ud \xi + \frac{1}{\pi} \sin(2x) \int_0^{\sqrt{2}} \xi^2 \sin(x \xi^2) \: \ud \xi \\
    & = - \frac{1}{\pi} \cos(2x) \int_0 ^{\sqrt{2}} \frac{\sin(x \xi^2)}{x} \: \ud \xi + \frac{1}{\pi} \sin(2x) \int_0 ^{\sqrt{2}} \frac{\cos(x \xi^2)}{x} \: \ud \xi,
  \end{align*}
  where we have applied the change of variables $\xi' = \sqrt{2 - \xi}$ and an integration by parts.
  By using the \textit{Fresnel integrals} (see, e.g.,~\cite[7.2 (iii)]{olver.lozier.ea:10:nist})
  \begin{align*}
    \mathcal{S}(z) = \int_0^{z} \sin\left(\frac{\pi t^2}{2}\right) \ud t \quad \text{and} \quad
    \mathcal{C}(z) = \int_0^z \cos\left(\frac{\pi t^2}{2}\right) \ud t,
	\end{align*}
	the above expression can be further simplified via the change of variables $\xi' = \sqrt{2|x|/\pi}
	\: \xi$ to
	\begin{align*}
	\mathcal{F}^{-1}\big(\sqrt{\widehat{f}\:}\:\:\big)(x)
	= (8\pi)^{-1/2} |x|^{-3/2} \left(-\cos(2|x|)\mathcal{S}\left(\frac{2\sqrt{|x|}}{\sqrt{\pi}}\right) + \sin(2|x|) \mathcal{C}\left( \frac{2 \sqrt{|x|}}{\sqrt{\pi}} \right) \right), \quad x \in \R.
	\end{align*}
  This proves the claim in \eqref{E:Claim}; see also Figure~\ref{Fig:fresnel} for some plots.
\end{example}

\begin{figure}[htbp]
	\centering
\begin{center}
\begin{tikzpicture}[scale = 0.8, transform shape]
\begin{axis}
[
		axis lines = left,
		xtick = {-8,-4,0,4,8},
		ytick = {0, 0.159, 0.300, 0.5},
		yticklabels = {$0$, $\frac{1}{2\pi}$, $\frac{2 \sqrt{2}}{3\pi}$, $\sfrac{1}{2}$},
		legend style = {at = {(1.0,0.9)},anchor = north},
		legend style = {draw = none, legend cell align = left},
		xmax = 10,
		ymax = 0.55,
		]

		\addplot[domain = -9:9,red,thick] coordinates {
			(-8.001,0.002432823286953949)
			(-7.901,0.0025438737134728592)
			(-7.800999999999999,0.0026079539418625837)
			(-7.701,0.0026213320142579343)
			(-7.600999999999999,0.0025821550165559203)
			(-7.5009999999999994,0.0024906227294374437)
			(-7.401,0.00234909155474499)
			(-7.300999999999999,0.0021621013460149494)
			(-7.201,0.0019363202666636112)
			(-7.100999999999999,0.001680405582745344)
			(-7.0009999999999994,0.0014047812821759637)
			(-6.901,0.0011213365012638092)
			(-6.800999999999999,0.0008430518273667448)
			(-6.701,0.0005835635245753582)
			(-6.600999999999999,0.00035667848826553106)
			(-6.5009999999999994,0.00017585516836534406)
			(-6.401,0.000053667711653983396)
			(-6.300999999999999,1.2720726207976823e-6)
			(-6.201,0.00002789375686314378)
			(-6.100999999999999,0.00014035713443167496)
			(-6.0009999999999994,0.0003426758564225485)
			(-5.901,0.0006357228116034111)
			(-5.800999999999999,0.0010169962789705433)
			(-5.700999999999999,0.001480496498286727)
			(-5.600999999999999,0.00201672384838161)
			(-5.5009999999999994,0.0026128062689099543)
			(-5.401,0.003252759582195788)
			(-5.300999999999999,0.003917880082258694)
			(-5.200999999999999,0.004587264287097178)
			(-5.100999999999999,0.005238446237433076)
			(-5.0009999999999994,0.0058481383164472895)
			(-4.901,0.006393057408427283)
			(-4.800999999999999,0.006850814454606368)
			(-4.700999999999999,0.007200842238934415)
			(-4.600999999999999,0.007425333669740383)
			(-4.5009999999999994,0.007510161023005499)
			(-4.401,0.007445745666168137)
			(-4.300999999999999,0.007227847750648969)
			(-4.200999999999999,0.006858246282234315)
			(-4.100999999999999,0.006345281857780303)
			(-4.0009999999999994,0.005704237171131473)
			(-3.900999999999999,0.004957534087349536)
			(-3.8009999999999993,0.004134730579773882)
			(-3.7009999999999996,0.0032723060091316676)
			(-3.600999999999999,0.0024132289632433244)
			(-3.5009999999999994,0.001606308014070545)
			(-3.400999999999999,0.0009053321133675948)
			(-3.3009999999999993,0.00036801375460334134)
			(-3.2009999999999987,0.00005475428626489342)
			(-3.100999999999999,0.000027256678591737287)
			(-3.0009999999999994,0.0003470164358733941)
			(-2.900999999999999,0.00107372603437282)
			(-2.8009999999999993,0.0022636320930730014)
			(-2.7009999999999987,0.003967887314026259)
			(-2.600999999999999,0.006230940950866248)
			(-2.5009999999999994,0.009089012098190241)
			(-2.400999999999999,0.012568689394647621)
			(-2.3009999999999993,0.01668569878780381)
			(-2.2009999999999987,0.021443877844817775)
			(-2.100999999999999,0.02683439077171419)
			(-2.0009999999999994,0.03283521292278903)
			(-1.900999999999999,0.03941090727057421)
			(-1.8009999999999993,0.04651270822547646)
			(-1.7009999999999987,0.05407892052761322)
			(-1.600999999999999,0.06203563288655326)
			(-1.5009999999999994,0.07029773783704628)
			(-1.400999999999999,0.07877024113784453)
			(-1.3009999999999993,0.08734983619528293)
			(-1.2009999999999987,0.09592671166698348)
			(-1.100999999999999,0.10438655380553113)
			(-1.0009999999999994,0.11261269943059007)
			(-0.9009999999999989,0.12048839083985655)
			(-0.8009999999999993,0.1278990806243089)
			(-0.7009999999999987,0.13473473234750544)
			(-0.6009999999999991,0.14089206245118208)
			(-0.5009999999999994,0.14627666958971366)
			(-0.4009999999999989,0.15080499986313772)
			(-0.30099999999999927,0.1544061000609773)
			(-0.20099999999999874,0.15702311595635107)
			(-0.10099999999999909,0.15861449877440564)
			(-0.0009999999999994458,0.1591548900402547)
			(0.0990000000000002,0.1586356629000162)
			(0.19900000000000162,0.15706510649154093)
			(0.29900000000000126,0.15446824878921397)
			(0.3990000000000009,0.15088632231941124)
			(0.49900000000000055,0.14637588599584814)
			(0.5990000000000002,0.1410076248166563)
			(0.6990000000000016,0.1348648570674698)
			(0.7990000000000013,0.12804178577404596)
			(0.8990000000000009,0.12064153725409628)
			(0.9990000000000006,0.11277403456977432)
			(1.0990000000000002,0.10455375735150653)
			(1.1990000000000016,0.0960974417591414)
			(1.2990000000000013,0.08752177521547783)
			(1.399000000000001,0.07894113997833622)
			(1.4990000000000006,0.070465457639247)
			(1.599000000000002,0.06219818331775087)
			(1.6990000000000016,0.054234493765639914)
			(1.7990000000000013,0.046659707944379146)
			(1.899000000000001,0.039547972060097115)
			(1.9990000000000006,0.03296123372682002)
			(2.099000000000002,0.026948522091409535)
			(2.1990000000000016,0.021545542616359)
			(2.2990000000000013,0.016774587008063356)
			(2.399000000000001,0.012644750725964176)
			(2.4990000000000006,0.009152442831703102)
			(2.599000000000002,0.006282165842513605)
			(2.6990000000000016,0.004007536924909025)
			(2.7990000000000013,0.00229251636350236)
			(2.899000000000001,0.0010928048962725856)
			(2.9990000000000006,0.00035736831968642875)
			(3.099000000000002,0.00003004579755997701)
			(3.1990000000000016,0.00005119758276101201)
			(3.2990000000000013,0.0003593483706586067)
			(3.399000000000001,0.0008927842019301681)
			(3.4990000000000006,0.0015910636407466753)
			(3.599000000000002,0.0023964077628005)
			(3.6990000000000016,0.0032549381598431354)
			(3.7990000000000013,0.004117737545113351)
			(3.899000000000001,0.004941713452235223)
			(3.9990000000000006,0.005690251772669383)
			(4.099000000000002,0.006333653282109798)
			(4.199000000000002,0.006849352673408142)
			(4.299000000000001,0.007221925758074765)
			(4.399000000000001,0.007442896247300105)
			(4.4990000000000006,0.007510358720626126)
			(4.599000000000002,0.007428438900773149)
			(4.699000000000002,0.007206616066218059)
			(4.799000000000001,0.006858935265808259)
			(4.899000000000001,0.006403138898051334)
			(4.9990000000000006,0.0058597481578703745)
			(5.099000000000002,0.0052511248416607)
			(5.199000000000002,0.004600543072479748)
			(5.299000000000001,0.003931298723252897)
			(5.399000000000001,0.003265881763437057)
			(5.4990000000000006,0.00262523354104581)
			(5.599000000000002,0.0020281072616718074)
			(5.699000000000002,0.001490545776069436)
			(5.799000000000001,0.0010254863827811726)
			(5.899000000000001,0.0006424978400591682)
			(5.9990000000000006,0.00034765030807791503)
			(6.099000000000002,0.00014351464796635594)
			(6.199000000000002,0.000029283517734928982)
			(6.299000000000001,1.0031415744061362e-6)
			(6.399000000000001,0.000051901585477685833)
			(6.4990000000000006,0.00017279692570908398)
			(6.599000000000002,0.00035256690231122526)
			(6.699000000000002,0.0005786605394831377)
			(6.799000000000001,0.0008376317967769518)
			(6.899000000000001,0.0011156755752281488)
			(6.9990000000000006,0.0013991473046539033)
			(7.099000000000002,0.0016750488265366993)
			(7.199000000000002,0.0019314652879756304)
			(7.299000000000001,0.0021579401886549533)
			(7.399000000000001,0.0023457784762423183)
			(7.4990000000000006,0.0024882705603940292)
			(7.599000000000002,0.0025808332023969515)
			(7.699000000000002,0.0026210663274543553)
			(7.799000000000001,0.0026087277946630874)
			(7.899000000000001,0.002545630948073522)
			(7.9990000000000006,0.0024354722734724237)
};
		\addlegendentry{$f(x) = (2\pi)^{-1} x^{-2}\sin^2 (x)$}
		\addplot[domain = -9:9,blue,thick,dashed] {0.25*max(2-abs(x),0)};
		\addlegendentry{$\widehat{f}(x) = 2^{-2}\max\left(2-|x|,0\right)$}

		\addplot[smooth, thick] coordinates {
			(-8.001,0.003817069187384621)
			(-7.901,0.004948077138048389)
			(-7.800999999999999,0.005957811046570557)
			(-7.701,0.006801808668466622)
			(-7.600999999999999,0.0074404986050931095)
			(-7.5009999999999994,0.007840943784860886)
			(-7.401,0.007978394690811678)
			(-7.300999999999999,0.007837586127350618)
			(-7.201,0.007413718638216873)
			(-7.100999999999999,0.006713075435127593)
			(-7.0009999999999994,0.005753237591198206)
			(-6.901,0.004562873928589717)
			(-6.800999999999999,0.003181097034880563)
			(-6.701,0.0016563926804693992)
			(-6.600999999999999,0.000045146020877752114)
			(-6.5009999999999994,-0.001590196215914143)
			(-6.401,-0.0031832734225176566)
			(-6.300999999999999,-0.004666201063584027)
			(-6.201,-0.0059721237551125395)
			(-6.100999999999999,-0.007037843288640616)
			(-6.0009999999999994,-0.007806427222745081)
			(-5.901,-0.008229700640346704)
			(-5.800999999999999,-0.008270524026900735)
			(-5.700999999999999,-0.00790476405612329)
			(-5.600999999999999,-0.007122871354141625)
			(-5.5009999999999994,-0.005930989901804334)
			(-5.401,-0.004351536364332931)
			(-5.300999999999999,-0.0024232039259575767)
			(-5.200999999999999,-0.00020036367237724425)
			(-5.100999999999999,0.002248143367067824)
			(-5.0009999999999994,0.004840809376747679)
			(-4.901,0.007485822677447826)
			(-4.800999999999999,0.010083934145729704)
			(-4.700999999999999,0.012531751213342676)
			(-4.600999999999999,0.014725365775906895)
			(-4.5009999999999994,0.01656420714088009)
			(-4.401,0.017954999093295454)
			(-4.300999999999999,0.018815691755267223)
			(-4.200999999999999,0.019079234553318915)
			(-4.100999999999999,0.01869705655185513)
			(-4.0009999999999994,0.017642124789483123)
			(-3.900999999999999,0.015911460050341306)
			(-3.8009999999999993,0.013528002550263853)
			(-3.7009999999999996,0.010541737007346405)
			(-3.600999999999999,0.007030007049935302)
			(-3.5009999999999994,0.0030969723178389673)
			(-3.400999999999999,-0.0011278127494024957)
			(-3.3009999999999993,-0.005491692343202579)
			(-3.2009999999999987,-0.009822041117901928)
			(-3.100999999999999,-0.013930173136600981)
			(-3.0009999999999994,-0.017615940315939035)
			(-2.900999999999999,-0.02067290843983226)
			(-2.8009999999999993,-0.022893977332707104)
			(-2.7009999999999987,-0.02407729388733329)
			(-2.600999999999999,-0.024032292978945078)
			(-2.5009999999999994,-0.02258569239072328)
			(-2.400999999999999,-0.019587264108344292)
			(-2.3009999999999993,-0.014915205938662947)
			(-2.2009999999999987,-0.008480944428455815)
			(-2.100999999999999,-0.00023321239116280357)
			(-2.0009999999999994,0.009838738289648315)
			(-1.900999999999999,0.02170290596748294)
			(-1.8009999999999993,0.0352833821081619)
			(-1.7009999999999987,0.05046033000543576)
			(-1.600999999999999,0.06707113001279474)
			(-1.5009999999999994,0.08491268888984531)
			(-1.400999999999999,0.10374487729202315)
			(-1.3009999999999993,0.1232950263870252)
			(-1.2009999999999987,0.14326338307207942)
			(-1.100999999999999,0.16332939426840914)
			(-1.0009999999999994,0.1831586651906907)
			(-0.9009999999999989,0.20241041514060143)
			(-0.8009999999999993,0.2207452379420889)
			(-0.7009999999999987,0.23783296316315466)
			(-0.6009999999999991,0.2533604091321415)
			(-0.5009999999999994,0.26703881965624665)
			(-0.4009999999999989,0.2786107833024671)
			(-0.30099999999999927,0.2878564469369221)
			(-0.20099999999999874,0.2945988535866673)
			(-0.10099999999999909,0.2987082580668381)
			(-0.0009999999999994458,0.3001053015279998)
			(0.0990000000000002,0.29876295731208347)
			(0.19900000000000162,0.2947071943387846)
			(0.29900000000000126,0.28801633967538115)
			(0.3990000000000009,0.27881915791927214)
			(0.49900000000000055,0.2672917004749534)
			(0.5990000000000002,0.25365301168263127)
			(0.6990000000000016,0.23815981005052952)
			(0.7990000000000013,0.22110029063378403)
			(0.8990000000000009,0.20278721807531874)
			(0.9990000000000006,0.1835504982981551)
			(1.0990000000000002,0.1637294297890097)
			(1.1990000000000016,0.14366484248546818)
			(1.2990000000000013,0.12369133330027068)
			(1.399000000000001,0.1041298022988072)
			(1.4990000000000006,0.08528048268668975)
			(1.599000000000002,0.06741664143717203)
			(1.6990000000000016,0.05077910613156309)
			(1.7990000000000013,0.03557174808422779)
			(1.899000000000001,0.02195802288578886)
			(1.9990000000000006,0.010058638029713459)
			(2.099000000000002,-0.00004961573703946895)
			(2.1990000000000016,-0.008333866267748716)
			(2.2990000000000013,-0.014804025099610152)
			(2.399000000000001,-0.019510577428420127)
			(2.4990000000000006,-0.022541387877197273)
			(2.599000000000002,-0.024017638964368225)
			(2.6990000000000016,-0.024089041207078127)
			(2.7990000000000013,-0.022928471252533566)
			(2.899000000000001,-0.020726206866197184)
			(2.9990000000000006,-0.017683934734407606)
			(3.099000000000002,-0.014008708744824414)
			(3.1990000000000016,-0.009907032746353366)
			(3.2990000000000013,-0.005579232983131978)
			(3.399000000000001,-0.001214271824182093)
			(3.4990000000000006,0.0030148634064358954)
			(3.599000000000002,0.006955084108966468)
			(3.6990000000000016,0.01047634847302212)
			(3.7990000000000013,0.013473972100250975)
			(3.899000000000001,0.01587006681208563)
			(3.9990000000000006,0.017614100997363203)
			(4.099000000000002,0.01868260197557976)
			(4.199000000000002,0.019078046526706048)
			(4.299000000000001,0.018827009192965684)
			(4.399000000000001,0.017977658505876553)
			(4.4990000000000006,0.016596708356281065)
			(4.599000000000002,0.01476594485514683)
			(4.699000000000002,0.012578457915475618)
			(4.799000000000001,0.010134711253107408)
			(4.899000000000001,0.007538584536305954)
			(4.9990000000000006,0.004893517136095283)
			(5.099000000000002,0.002298874604146539)
			(5.199000000000002,-0.00015335297806290054)
			(5.299000000000001,-0.0023814267769490752)
			(5.399000000000001,-0.0043162328187769305)
			(5.4990000000000006,-0.0059030967602637895)
			(5.599000000000002,-0.007103003959947052)
			(5.699000000000002,-0.007893210283957092)
			(5.799000000000001,-0.008267250120896385)
			(5.899000000000001,-0.008234368175256704)
			(5.9990000000000006,-0.00781842011188186)
			(6.099000000000002,-0.0070563034614563045)
			(6.199000000000002,-0.005995993883001305)
			(6.299000000000001,-0.0046942725326995964)
			(6.399000000000001,-0.0032142376401147967)
			(6.4990000000000006,-0.0016226972943811579)
			(6.599000000000002,0.000012459131697669528)
			(6.699000000000002,0.0016248164496881403)
			(6.799000000000001,0.0031518276795062636)
			(6.899000000000001,0.004536967292258031)
			(6.9990000000000006,0.00573157579874873)
			(7.099000000000002,0.006696341263233783)
			(7.199000000000002,0.0074023782331208285)
			(7.299000000000001,0.007831880379207368)
			(7.399000000000001,0.007978339253663794)
			(7.4990000000000006,0.007846337419529361)
			(7.599000000000002,0.007450939234636234)
			(7.699000000000002,0.006816716278844069)
			(7.799000000000001,0.005976456344665158)
			(7.899000000000001,0.004969614711354385)
			(7.9990000000000006,0.003840573784180821)
};
		\addlegendentry{$\mathcal{F}^{-1}\big(\sqrt{\widehat{f}\:}\:\:\big)(x)$}

\end{axis}

\end{tikzpicture}
\end{center}
	\caption{One example that $\mathcal{F}^{-1}\big(\sqrt{\widehat{f}\:}\:\:\big)(x)$ assumes both
		positive and negative values.}
	\label{Fig:fresnel}
\end{figure}
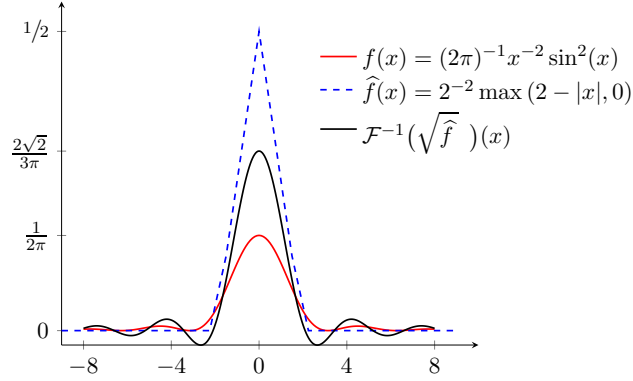

On the other hand, the next example shows that the spectral density given in Example~\ref{Ex:Finv}
can be easily handled by our condition -- $\Upsilon_{\alpha}(0)<\infty$, i.e.,~\eqref{E:Dalang-a0}.

\begin{example}\label{Ex:Finv-d}
  From the one-dimension case in Example~\ref{Ex:Finv}, one can construct a $d$-dimensional
  counterpart: Let $f_1$ and $\widehat{f}_1$ be the $f$ and $\widehat{f}$, respectively, in
  Example~\ref{Ex:Finv}. Then define
	\begin{align*}
    f_d(x) := \prod_{i=1}^d f_1(x_i), \quad x\in\R^d, \quad \text{and hence} \quad
    \widehat{f}_d(\xi) := \prod_{i=1}^d \widehat{f}_1(\xi_i), \quad \xi\in\R^d.
	\end{align*}
  It is straightforward to verify that
	\begin{align*}
    \Upsilon_\alpha(0)
    = C \int_{\R^d} \frac{\prod_{i = 1}^d\max\{2-|\xi_i|,0\}}{|\xi|^{2(1-\alpha)}} \ud \xi
    \le C \int_{|\xi|\le 2\sqrt{d}} \frac{2^d}{|\xi|^{2(1-\alpha)}} \ud \xi
    = C \int_{0}^{2\sqrt{d}} \frac{r^{d-1}}{r^{2(1-\alpha)}} \ud r.
	\end{align*}
  Hence, if $\alpha > 1-d /2$, then $\Upsilon_\alpha(0)<\infty$.
\end{example}


The next example illustrates the delicacy of choosing the right branches for the square root
in~\eqref{E:TZspec}.

\begin{example}\label{Ex:sqrt}
  Let $f$ and $g$ be given as Example~\ref{Ex:nonnegative}. In this case, $\widehat{f}(\cdot)$ is
  only nonnegative (not strictly positive) with infinitely many zeros. Hence, when taking the square
  root of $\widehat{f}(\xi)$ as in~\eqref{E:TZspec}, one needs to wisely select the correct positive
  and negative branches: (1) Clearly, the signed version $\sqrt{\widehat{f}(\xi)} =
  \xi^{-1}\sin(\xi)$ is preferable since its inverse Fourier transform can be easily computed, which
  is equal to $g(x)$. Moreover, because this inverse Fourier transform $g(x)$ is nonnegative, the
  absolute value signs in~\eqref{E:TZspec} do not pose any additional restrictions. (2) However, if
  one chooses the positive branches, namely, $\sqrt{\widehat{f}(\xi)} =
  \left|\xi^{-1}\sin(\xi)\right|$, then it is not clear how to compute its Fourier transform. In
  general, some bad choices of the positive/negative branches may make the conditions
  in~\eqref{E:TZspec} fail. For example, such choice may turn $\sqrt{\widehat{f}(\xi)}$ into a
  distribution, and then taking the absolute value of a distribution (unless it is a measure) may be
  problematic. Another issue that may arise is when $\sqrt{\widehat{f}(\xi)}$ is a well-defined
  function, taking on both positive and negative values and after taking the absolute value, the
  integral in~\eqref{E:TZspec} may blow up.
\end{example}

\subsection{Various initial conditions}\label{SS:Init}

In this part, we give some concrete examples of initial conditions.

\begin{example}[$L^\infty(\R^d)$ initial condition]\label{Ex:BddInit}
  We emphasize that if the initial condition, $\mu$, is deterministic and is such that $\mu(\ud x) =
  \varphi(x)\ud x$ with $\varphi\in L^\infty(\R^d)$, then all conditions related to
  $\mathcal{G}_\rho(\cdot)$ in both Theorems~\ref{T:Lrho} and~\ref{T:Main} are trivially satisfied.
  To be more precise, both Conditions~\eqref{E:G} and~\eqref{E:InitData} hold because
  \begin{align*}
    \mathcal{G}_\rho(t;|\varphi|) \le \Norm{\varphi}_{L^\infty(\R^d)}^2 \Norm{\rho}_{L^1 (\R^d)}< \infty
    \quad \text{uniformly for all $t\ge 0$.}
  \end{align*}
\end{example}

\begin{example}[Delta initial condition]\label{Ex:delta}
  In this example, we study the case when the initial condition, $\mu$, is the Dirac delta measure
  at zero, namely, $\delta_0$. Let $\rho$ be a nonnegative $L^1 (\R^d)$ function. Since
  \begin{align*}
    \mathcal{G}_\rho(t;|\delta_0|)
    = \int_{\R^d} G(t,x)^2 \rho(x) \ud x
    \le G(t,0)^2 \Norm{\rho}_{L^1 (\R^d)},
    \quad \text{for all $t>0$,}
  \end{align*}
  we see that both conditions~\eqref{E:G} and~\eqref{E:InitData} are satisfied. In particular,
  $\limsup_{t>0} \mathcal{G}_\rho(t;\delta_0) = 0$.
\end{example}

\begin{example}[More initial conditions not in $L_\rho^2 (\R^d)$]\label{Ex:Riesz}
  In this example, we study the case when $\mu(\ud x) = |x|^{-\alpha}\ud x$ for some $\alpha\in
  (0,d)$. It is clear that when $\alpha\in (d/2,d)$, $\mu\not\in L_\rho^2 (\R^d)$. However, in this
  case, we have
  \begin{align*}
    J_0(t,x) = \left(G(t,\cdot)* |\cdot|^{-\alpha}\right)(x)
    \le \left(G(t,\cdot)* |\cdot|^{-\alpha}\right)(0).
  \end{align*}
  On the other hand,
  \begin{align*}
    \left(G(t,\cdot)* |\cdot|^{-\alpha}\right)(0)
    = \frac{2\pi^{d/2}}{\Gamma(d/2)}\times (2\pi t)^{-d/2} \int_0^\infty e^{-\frac{r^2}{2t}} r^{-\alpha + d-1} \ud r = C_* t^{-\alpha/2},
  \end{align*}
  with $C_* = 2^{-\alpha/2}\Gamma\left((d-\alpha)/2\right)/\Gamma(d/2)$, which implies that
  \begin{align*}
    \mathcal{G}_\rho\left(t;|\cdot|^{-\alpha}\right)
    \le \int_{\R^d} J_0^2 (t,0) \rho(x) \ud x
    = C_*^2 t^{-\alpha} \Norm{\rho}_{L^1 (\R^d)},
    \quad \text{for all $t>0$.}
  \end{align*}
  Therefore, we see that both conditions~\eqref{E:G} and~\eqref{E:InitData} are satisfied.
\end{example}

The following proposition shows that for initial conditions with unbounded tails,
condition~\eqref{E:G} may hold while condition~\eqref{E:InitData} may fail.

\begin{proposition}\label{P:growtail}
  Suppose that $\rho(x) = \exp(-|x|)$, which is an admissible weight function. Let the initial
  condition $\mu$ be given as $\mu(\ud x) = |x|^\alpha \ud x$ with $\alpha>0$. Then for some
  constants $C, C'>0$ that depend on $d$ and $\alpha$, it holds that
  \begin{align}\label{E:growtail}
    C' (1+t^\alpha) \le \mathcal{G}_\rho(t;|\mu|) \le C (1+t^\alpha), \qquad \text{for all $t>0$}.
  \end{align}
  In particular, this implies that condition~\eqref{E:G} is satisfied, but condition~\eqref{E:InitData} fails.
\end{proposition}
\begin{proof}
  Notice that by scaling arguments, $\mathcal{G}_\rho(t;|\mu|)$ is equal to
  \begin{align*}
    &   \int_{\R^d} \left(\int_{\R^d}G(t,x-y) |y|^\alpha \ud y\right)^2 e^{-|x|} \ud x
      = \int_{\R^d}t^{\alpha+d/2} \left(\int_{\R^d}G\left(1,\xi-z\right) |z|^\alpha \ud z\right)^2 e^{-\sqrt{t}|\xi|} \ud \xi.
  \end{align*}
  In the following, let $C_d$, $C_\alpha$, $C_\alpha'$, $C_{\alpha, d}$ and $C'_{\alpha, d}$ be
  generic constants that may depend on $\alpha $ and $d$ and may change their value at each
  appearance. \medskip

  \noindent\textit{Upper bound:~} Because
  \begin{align*}
    \int_{\R^d}G\left(1,z\right) |\xi-z|^\alpha \ud z
    \le C_\alpha \int_{\R^d}G\left(1,z\right) \left(|\xi|^\alpha + |z|^\alpha\right) \ud z
    \le C_\alpha' \left(1+|\xi|^\alpha\right),
  \end{align*}
  we see that
  \begin{align*}
    \mathcal{G}_\rho(t;\mu)
    & \le C_\alpha \int_{\R^d}t^{\alpha+d/2} \left(1+|\xi|^{2\alpha}\right) e^{-\sqrt{t}|\xi|} \ud \xi
      = C_{\alpha, d} \left(t^\alpha \Gamma(d) + \Gamma(d+2\alpha)\right)
      = C'_{\alpha, d} (1+ t^\alpha)<\infty,
  \end{align*}
  which proves the upper bound in~\eqref{E:growtail}. \medskip

  \noindent\textit{Lower bound:~} Now we prove the lower bound in~\eqref{E:growtail}. Indeed,
  \begin{align*}
    \MoveEqLeft \int_{\R^d}G\left(1,z\right) |\xi-z|^\alpha \ud z
      \ge \int_{\R^d}G\left(1,z\right) \big| |\xi| - |z| \big|^\alpha \ud z \\
    & \ge {C_d} \int_0^\infty \big| |\xi| - x \big|^\alpha e^{-\frac{x^2}{2}}x^{d-1}\ud x
      \ge C_d \int_1^2 \big| |\xi| - x \big|^\alpha \ud x
      = \frac{C_d}{1+\alpha} \psi\left(|\xi|\right),
  \end{align*}
  where, by considering three cases, we have
  \begin{align*}
    \psi(r)
    & =
    \begin{cases}
      (2-r)^{\alpha+1} - (1-r)^{\alpha+1} & \text{if $0<r<1$},        \\
      (2-r)^{\alpha+1} + (r-1)^{\alpha+1} & \text{if $1\le r \le 2$}, \\
      (r-1)^{\alpha+1} - (r-2)^{\alpha+1} & \text{if $r > 2$},
    \end{cases}
  \end{align*}
  which is equal to $\text{sgn}(2-r)|r-2|^{\alpha+1} + \text{sgn}(r-1)|r-1|^{\alpha+1}$.
  We claim that
  \begin{align}\label{E:psi}
    \inf_{r\ge 0} \frac{\psi(r)}{\sqrt{1+r^{2\alpha}}}>0.
  \end{align}
  With~\eqref{E:psi}, we have that
  \begin{align*}
    \int_{\R^d}G\left(1,z\right) |\xi-z|^\alpha \ud z \ge C_{\alpha, d} \sqrt{1+|\xi|^{2\alpha}}.
  \end{align*}
  Then, by the same arguments as above for the upper bound, we obtain the lower bound
  in~\eqref{E:growtail}. It remains to prove~\eqref{E:psi}, which will be proved in three cases.

  When $r>2$, we see that
  \begin{align*}
    \frac{\psi(r)}{\sqrt{1+r^{2\alpha}}}
    & \ge C_\alpha \frac{(r-1)^{\alpha+1}-(r-2)^{\alpha+1}}{(1+r)^\alpha}
      \ge C_\alpha \frac{(r-1)^{\alpha}(r-1)-(r-1)^{\alpha}(r-2)}{(1+r)^\alpha} \\
    & = C_\alpha \left(\frac{r-1}{1+r}\right)^\alpha
      = C_\alpha \left(1-\frac{2}{1+r}\right)^\alpha \ge C_\alpha \left(1-\frac{2}{3}\right)^\alpha.
  \end{align*}
  Note that in the first inequality above, we have considered two cases: $2\alpha \ge 1$ and $2
  \alpha < 1$. When $2\alpha < 1$, we have used the concavity of $x^{2\alpha}$, namely, $ (1 +
  r^{2\alpha})^{2\alpha} / 2 \le ((1+r)/2)^{2\alpha}$; when $2\alpha \ge 1$, we have used the
  super-additivity of $x^{2\alpha}$: namely, that for $a,b>0$ that $ (a+b)^{2\alpha}\ge a^{2\alpha}
  + b^{2\alpha}$. Therefore, $\inf_{r>2} \frac{\psi(r)}{\sqrt{1+r^{2\alpha}}} >0$.

  When $r\in(1,2]$, elementary calculations show that the minimum of $\psi(r)$ is achieved at
  $r = 3/2$. Hence, $\inf_{r\in(1,2]} \frac{\psi(r)}{\sqrt{1+r^{2\alpha}}} \ge \frac{\psi(3/2)
}{\sqrt{1+4^\alpha}}>0$.

  Similarly, when $r\in (0,1]$, by differentiation, one finds that the function $\psi(r)$ is
  nonincreasing. Hence, the minimum is achieved at $r = 1$, $\inf_{r\in(0,1]}
  \frac{\psi(r)}{\sqrt{1+r^{2\alpha}}} \ge \frac{\psi(1)}{\sqrt{2}}>0$.

  Combing the above three cases proves~\eqref{E:psi} and hence, Proposition~\ref{P:growtail}.
\end{proof}

\subsection{Bessel and other related kernels}\label{SS:Bassel}

In this part, we will make some explicit computations for Bessel and related kernels.

\begin{example}[Bessel kernel]\label{Ex:Bessel}
  Let $f_s$ denote the Bessel kernel with a strictly positive parameter $s>0$. It is known that
  (see, e.g., Section 1.2.2 of~\cite{grafakos:14:modern})
  \begin{enumerate}
    \item $f_s(x)>0$ for all $x\in\R^d$ and $ \Norm{f_s}_{L^1 (\R^d)} = 1$;
    \item there exists a constant $C(s,d)>0$ such that $f_s(x) \le C(s,d) \exp(-|x|/2)$ for $|x|\ge
      2$.
    \item there exists a constant $c(s,d)>0$ such that
      \begin{gather*}
        \frac{1}{c(s,d)} \le \frac{f_s(x)}{H_s(x)} \le c(s,d) \quad \text{for} \quad |x| \le 2,
        \shortintertext{with}
        H_s(x) =
        \begin{cases}
          |x|^{s-d} + 1 + O(|x|^{s-d+2}) & \text{for $0<s<d$,} \\
          \log\left(\frac{2}{|x|}\right) + 1 + O(|x|^{2}) & \text{for $s = d$,} \\
           1 + O(|x|^{s-d}) & \text{for $s>d$;}
        \end{cases}
        \end{gather*}
      \item the Fourier transform of $f_s$ is strictly positive:
        \begin{equation}\label{E:FB}
          \mathcal{F}f_s(\xi) = \frac{1}{(1+|\xi|^2)^{s/2}}.
        \end{equation}
  \end{enumerate}
  Note that one can use~\eqref{E:FB} as the definition of the Bessel kernel. Properties 1 and 4
  ensure that $f_s$ is a nonnegative and nonnegative-definite tempered measure for all $s>0$.
\end{example}

\begin{example}[Mat\'ern class of correlation functions]
  The \textit{Mat\'ern class of correlation functions} has been widely used in spatial statistics;
  one may check the recent work~\cite{loh.sun.ea:21:on} for references. Following Section 2.10
  of~\cite{stein:99:interpolation}, this class of correlation functions is given by
  \begin{align}\label{E:Matern}
    K(x) = \phi \cdot (\alpha |x|)^\nu \mathcal{K}_\nu\left(\alpha |x|\right), \quad \text{for
    $x\in\R^d$ with $\phi>0$, $\alpha>0$, $\nu >0$},
  \end{align}
  where $\mathcal{K}_\nu(\cdot)$ is the modified Bessel function of second type, and $\alpha$ and
  $\nu$ refer to the \textit{scaling and smoothness parameters}, respectively. From the inversion
  formula (see p. 46 \textit{ibid.}), one sees that
  \begin{align*}
    \mathcal{F}K(\xi) = (2\pi)^{d} \mathcal{F}^{-1} K(\xi) = (2\pi)^d f(|\xi|)
    \quad \text{with} \quad
    f(\xi) = \frac{2^{\nu-1} \phi \Gamma\left(\nu +d/2\right)\alpha^{2\nu}}{\pi^{d/2}\left(\alpha^2+|\xi|^2\right)^{\nu+d/2}}, \quad \xi\in\R^d.
  \end{align*}
  Comparing the above expression with~\eqref{E:FB}, we see that the class of Bessel kernels $f_s$,
  with $s>d-2$ and $d\ge 3$, includes the Mat\'ern class~\eqref{E:Matern} as a special case under
  the following choice of parameters:
  \begin{align*}
    \alpha = 1,\quad \nu = (s-d)/2,\quad \text{and} \quad \phi = 2^{(2-d-s)/2} \pi^{-d/2} \Gamma(s /2)^{-1}.
  \end{align*}
  Note that the requirement of the smoothness parameter $\nu>0$ for the Mat\'ern class corresponds
  to the case of the Bessel kernel with $s>d$.
\end{example}

The following proposition shows what conditions~\eqref{E:Dalang-a0},~\eqref{E:Dalang-00},
and~\eqref{E:Dalang-H} reduce to for the Bessel kernel as the correlation function in terms of its
parameters.

\begin{proposition}[Bessel kernel as correlation function]\label{P:Bessel}
  If the correlation function $f$ is given by the Bessel kernel $f_s(\cdot)$ with $s>0$ defined in
  Example~\ref{Ex:Bessel}, then
  \begin{align}\label{E:UpsionB}
    \Upsilon_{\alpha}(0) = \frac{\Gamma\left(\frac{d}{2} -1 + \alpha\right) \Gamma\left(\frac{s-d}{2} + 1 - \alpha\right)}{2^d \pi^{d/2} \: \Gamma(d/2)\:\Gamma\left(s/2\right)}
    \qquad \text{for all $s>d-2(1-\alpha)>0$,}
  \end{align}
   and in particular when $\alpha = 0$,~\eqref{E:UpsionB} simplifies to the following:
  \begin{equation}\label{E:Upsilon_0}
     \Upsilon(0) = \frac{\Gamma\left(\frac{2+s-d}{2}\right)}{2^{d-1} \pi^{d/2} (d-2)\Gamma(s/2)}
     \qquad \text{for all $s>d-2>0$.}
  \end{equation}
  In addition,
  \begin{gather}\label{E:HaB}
    \mathcal{H}_\alpha(t) <\infty \quad \forall t>0 \quad \Longleftrightarrow \quad
    0<\alpha<\frac{1}{2} - \frac{(d-s)_+}{4}\quad \text{and} \quad s>d-2>0,
  \end{gather}
  where $a_+ \coloneqq \max(a,0)$. Moreover, for $\alpha\in(0,1/2)$, we have the following asymptotic
  behavior of
  $\mathcal{H}_\alpha(t)$ at $t\to 0$: 
  \begin{align}\label{E:HaBS}
    \hspace{-1em}
    \mathcal{H}_\alpha(t) = \left\{
      \begin{array}{lcl}
        \begin{aligned}
          \dfrac{\pi^{d/2} \Gamma\left((d-s)/2\right)}{\left((s-d)/2+1-2\alpha\right)\Gamma(d/2)}  t^{(s-d)/2+1-2\alpha}                  \\[1em]
          \quad - \dfrac{\pi^{d/2}\Gamma\left((s-d)/2\right)}{(1-2\alpha)\Gamma(s/2)}              t^{1-2\alpha}                          \\[1em]
          +                                                                                        O\left(t^{(s-d)/2+2(1-\alpha)}\right), \\[1.2em]
        \end{aligned}
        &
        \begin{gathered}
         d-2<s<d    \\
         \text{and} \\
         \alpha < \frac{1}{2} -\frac{1}{4}(d-s)
       \end{gathered} & \text{\em(\ref{E:HaBS}-a)} \\
        \begin{aligned}
            \frac{\pi ^{d/2}}{(1-2 \alpha) \Gamma \left(d/2\right)}t^{1-2\alpha}\log\left(\frac{1}{t}\right)                                    \\[1em]
           +\frac{\pi ^{d/2} \left(1- (1-2 \alpha) \left[\psi(d/2)+2\gamma\right]\right)}{(1-2 \alpha)^2 \Gamma \left(d/2\right)}t^{1-2 \alpha} \\[1em]
           +O\left(t^2\log(t)\right),                                                                                                           \\[1.2em]
        \end{aligned}                                                                                                           & s=d     & \text{\em(\ref{E:HaBS}-b)} \\ [1em]
        \dfrac{\pi^{d/2}\Gamma\left((s-d)/2\right)}{(1-2\alpha)\Gamma(s/2)}t^{1-2\alpha} + O\left(t^{(s-d)/2+1-2\alpha}\right), & d<s<d+2 & \text{\em(\ref{E:HaBS}-c)} \\ [1.5em]
        \dfrac{\pi^{d/2}}{(1-2\alpha)\Gamma\left(d/2+1\right)}t^{1-2\alpha}  + O\left(t^{2(1-\alpha)}\log(t)\right),            & s=d+2   & \text{\em(\ref{E:HaBS}-d)} \\ [1.5em]
        \dfrac{\pi^{d/2}}{(1-2\alpha)\Gamma\left(d/2+1\right)}t^{1-2\alpha}  + O\left(t^{2(1-\alpha)}\right),                   & s>d+2   & \text{\em(\ref{E:HaBS}-e)} \\ [1.5em]
      \end{array}
      \right.
  \end{align}
  where $\psi(x) = \frac{\ud}{\ud x}\log\Gamma(x)$ refers to the \emph{digamma function} and
  $\gamma\approx 0.57721$ to \emph{Euler's constant}; see, e.g., 5.2.2 and 5.2.3 on p. 136
  of~\cite{olver.lozier.ea:10:nist}.
\end{proposition}
\begin{proof}
  By the spherical coordinate integration formula and~\eqref{E:FB}, for all $\alpha\in[0,1)$,
  \begin{align*}
    \Upsilon_\alpha(0) = \left(2\pi\right)^{-d} \int_{\R^d} \frac{\ud\xi}{|\xi|^{2(1-\alpha)}\left(1+|\xi|^2\right)^{s/2}}
    = \left(2\pi\right)^{-d} C_d \int_0^\infty \frac{r^{d-1}}{r^{2(1-\alpha)}\left(1+r^2\right)^{s/2}} \ud r,
  \end{align*}
  where $C_d\coloneqq \frac{2\pi^{d/2}}{\Gamma\left(d/2\right)}$. Now by the change of variables
  $z = r^2 /(1+r^2)$, we can evaluate the above integral explicitly by transforming it to the Beta
  integral:
  \begin{align*}
    \int_0^\infty \frac{r^{d-1}}{r^{2(1-\alpha)}\left(1+r^2\right)^{s/2}} \ud r
     & = \frac{1}{2} \int_0^1 z^{d/2+\alpha-2}(1-z)^{(s-d)/2-\alpha}\ud z \\
     & = \frac{\Gamma\left(d/2-1+\alpha\right)\Gamma\left((s-d)/2+1-\alpha\right)}{2\Gamma(s/2)},
  \end{align*}
  which is finite provided that $s>d-2(1-\alpha)>0$. This proves~\eqref{E:UpsionB} and from this, we
  easily deduce~\eqref{E:Upsilon_0} by letting $\alpha = 0$ in~\eqref{E:UpsionB} and by applying the
  formula $\Gamma(z+1) = z\Gamma(z)$, which holds for $z\in \mathbb{C}$ such that $\Re(z)>0$.

  It remains to prove~\eqref{E:HaBS}, which then implies~\eqref{E:HaB}. From~\eqref{E:H} and by the
  spherical coordinate integration formula, for all $t>0$,
  \begin{align*}
    \mathcal{H}_{\alpha}(t)
      = \int_0^t \ud r \; r^{-2\alpha} \int_{\R^d} \ud \xi \; \frac{\exp(-r|\xi|^2)}{(1+|\xi|^2)^{s/2}}
    & = \frac{C_d}{2} \int_0^t \ud r \; r^{-2\alpha} \int_0^\infty \ud u \; \exp(-r u)(1+u)^{-s/2} u^{d/2-1} \\
    & \eqqcolon \frac{C_d \Gamma\left(d/2\right)}{2} \int_0^t \ud r \; r^{-2\alpha} I(r) = \pi^{d/2}\int_0^t \ud r \; r^{-2\alpha} I(r).
  \end{align*}
  By~\cite[13.4.4 on p.326]{olver.lozier.ea:10:nist}, $I(r)$ is equal to the \textit{confluent
  hypergeometric function}:
  \begin{align}\label{E:ConHyGe}
    I(r) =U\left(\frac{d}{2},\frac{2+d-s}{2},r\right).
  \end{align}
  By 18.2.18 -- 13.2.22 on p. 323 \textit{ibid.}, we see that \renewcommand{\arraystretch}{2}
  \begin{align*}
    I(r) = \left\{
      \begin{array}{lcl}
        \dfrac{\Gamma\left((d-s)/2\right)}{\Gamma(d/2)}r^{(s-d)/2}+\dfrac{\Gamma\left((s-d)/2\right)}{\Gamma(s/2)} + O\left(r^{(s-d)/2+1}\right) & d-2<s<d & \text{\small 18.2.18}, \\
        -\dfrac{1}{\Gamma(d/2)}\left(\log(r) + \psi(d/2) + 2\gamma \right) +O\left(r\log(r)\right)                                               & s=d     & \text{\small 18.2.19}, \\
        \dfrac{\Gamma\left((s-d)/2\right)}{\Gamma(s/2)} + O\left(r^{(s-d)/2}\right)                                                              & d<s<d+2 & \text{\small 18.2.20}, \\
        \dfrac{1}{\Gamma\left(d/2+1\right)} + O\left(r\log(r)\right)                                                                             & s=d+2   & \text{\small 18.2.21}, \\
        \dfrac{\Gamma\left((s-d)/2\right)}{\Gamma(s/2)} + O\left(r\right)                                                                        & s>d+2   & \text{\small 18.2.22}. \\
      \end{array}
      \right.
  \end{align*}
  Then integrating the right-hand side of the above expressions against $\pi^{d/2}r^{-2\alpha}\ud r$
  over $[0,t]$ gives the five cases in~\eqref{E:HaBS}. This completes the proof of
  Proposition~\ref{P:Bessel}.
\end{proof}

Similarly, one can use the Bessel kernel as the spectral density. In this case, we have the
following proposition:

\begin{proposition}[Bessel kernel as spectral density]\label{P:Bessel'}
  Suppose that the spectral density $\widehat{f}$ is given by the Bessel kernel $f_s(\cdot)$ defined
  in Example~\ref{Ex:Bessel}, or equivalently (see~\eqref{E:FB}), suppose that $f(x) =
  (1+|x|)^{-s/2}$ for $s>0$. Then
  \begin{align}\label{E:UpsionB'}
    \Upsilon_{\alpha}(0)
    = \frac{\Gamma\left(1 - \alpha\right) \Gamma\left(\alpha - 1 + s/2\right)}{2^{2d} \pi^{3d/2} \Gamma(d/2)\Gamma\left(s/2\right) }
    \qquad \text{for all $s>2(1-\alpha)>0$,}
  \end{align}
   and in particular when $\alpha = 0$,~\eqref{E:UpsionB'} simplifies to the following:
  \begin{equation}\label{E:Upsilon_0'}
     \Upsilon(0) = \frac{2^{1-2 d}\pi^{-3d/2}}{(s-2)\Gamma\left(d/2\right)}
     \qquad \text{for all $s>2$.}
  \end{equation}
  In addition,
  \begin{gather}\label{E:HaB'}
    \mathcal{H}_\alpha(t) <\infty \quad \forall t>0 \quad \Longleftrightarrow \quad
    0<\alpha<\frac{1}{2} \quad \text{and} \quad s>0.
  \end{gather}
 Moreover, for $\alpha\in(0,1/2)$, we have the following asymptotic
  \begin{align}\label{E:HaBS'}
    \mathcal{H}_\alpha(t) \sim \frac{t^{1-2\alpha}}{1-2\alpha}, \quad \text{as $t\downarrow 0.$}
  \end{align}
\end{proposition}
\begin{proof}
  By similar arguments as Proposition~\ref{P:Bessel}, we have that
	\begin{align*}
     \Upsilon_{\alpha}(0)
     & = (2\pi)^{-d} \int_{\R^d} \frac{f_s(\xi)}{|\xi|^{2(1-\alpha)}} \ud \xi
       = (2\pi)^{-2d} \int_{\R^d} \frac{\widehat{f}_s(\xi)}{|\xi|^{d-2(1-\alpha)}} \ud \xi                                                               \\
     & = (2\pi)^{-2d} \frac{2 \pi^{d/2}}{\Gamma\left(d/2\right)} \int_0^\infty \frac{r^{d-1}}{ \left(1+r^2\right)^{s / 2} \: r^{d - 2(1-\alpha)} } \ud r
       = \frac{\Gamma\left(1 - \alpha\right) \Gamma\left(\alpha - 1 + s/2\right)}{2^{2d} \pi^{3d/2} \Gamma(d/2)\Gamma\left(s/2\right) },
	\end{align*}
  which is finite provided $s>2(1-\alpha)$. This proves both~\eqref{E:UpsionB'}
  and~\eqref{E:Upsilon_0'}. As for~\eqref{E:HaB'},
	\begin{align*}
        \mathcal{H}_\alpha(t)
    & = \int_{0}^t \ud r \: r^{-2\alpha} \int_{\R^d} \ud \xi \: f_{s}(\xi) \exp\left(-r|\xi|^2\right)                                                    \\
    & = (2\pi)^{-d} \pi^{d/2} \int_{0}^t \ud r \: r^{-2\alpha} \int_{\R^d} \ud \xi \: \widehat{f}_{s}(\xi) \exp\left(-\frac{|\xi|^2}{4r}\right) r^{-d/2} \\
    & = (2\pi)^{-d} \pi^{d/2} \int_{0}^t \ud r \: r^{-2\alpha} \int_{\R^d} \ud \xi \: (1 + |\xi|^2)^{-s / 2} \exp\left(-\frac{|\xi|^2}{4r}\right) r^{-d/2}\\
    & = (2\pi)^{-d} \pi^{d/2} \frac{2 \pi^{d/2}}{\Gamma\left(d/2\right)} \int_{0}^t \ud r \: r^{-2\alpha - d/2} \int_0^{\infty} \ud z \: z^{d-1} (1 + z^2)^{-s/2} \exp\left(-\frac{z^2}{4r}\right),
\end{align*}
  where we have used Plancherel's theorem and the following identities:
	\begin{align*}
    \mathcal{F}\left(\exp(-|\cdot|^2)\right)(\xi) = \pi^{d/2} \exp\left(-4^{-1}|\xi|^2\right) \quad \text{and} \quad
    \mathcal{F}(f(a\cdot))(\xi) = a^{-d}\mathcal{F}f(\xi / a).
	\end{align*}
  Then, by the same arguments as Proposition~\ref{P:Bessel},
	\begin{align*}
        \mathcal{H}_\alpha(t)
     & = 2^{-d} \int_{0}^t \ud r \: r^{-2\alpha - d/2} I\left(\frac{1}{4r}\right),
      \quad \text{with $I(r) = U\left(\frac{d}{2}, \frac{2+d-s}{2}, r\right)$,}
	\end{align*}
  where $U$ is given in~\eqref{E:ConHyGe}. As $r \to \infty$, $I(r) \sim r^{-d/2}$ thus as $r \to
  0$, $I\left(\frac{1}{4r}\right) \sim (4r)^{d/2}$ (see 13.2.6 on p. 322
  of~\cite{olver.lozier.ea:10:nist}). Hence, the above integral behave as follows:
	\begin{align*}
     \mathcal{H}_\alpha(t)
     \sim \int_{0}^t \ud r \: r^{-2\alpha - \frac{d}{2}} (4r)^{d/2}
     = \int_{0}^t \ud r \: r^{-2\alpha}
     = \frac{t^{1-2\alpha}}{1-2\alpha},
	\end{align*}
  provided $\alpha \in (0, 1/2)$, which proves both~\eqref{E:HaB'} and~\eqref{E:HaBS'}.
\end{proof}

The necessity of the finiteness of $\Upsilon(0)$ excludes the Riesz kernel as a choice for the
spectral density. However, we can still construct a Riesz-type kernel which has polynomial growth at
the origin and polynomial decay at infinity, but with different rates, using
Propositions~\ref{P:Bessel} and~\ref{P:Bessel'}. This Riesz-type type kernel gives another example
of a kernel that is easily verifiable to be permissible under the conditions of our
Theorem~\ref{T:Main} above, while being demanding to verify using~\eqref{E:TZspec}; see
Example~\ref{Ex:Riesz-T} for more details.

\begin{example}[Riesz-type kernel]\label{Ex:Riesz-T}
  For $s_1, s_2\in (0,d)$, let $f_{s_1}$ and $f_{s_2}$ be Bessel kernels as in
  Example~\ref{Ex:Bessel}. Define
  \begin{align*}
    r(x)\coloneqq f_{s_1}(x) + \widehat{f}_{s_2}(x) \quad \text{or equivalently} \quad
    \widehat{r}(\xi)\coloneqq \widehat{f}_{s_1}(x) + f_{s_2}(x).
  \end{align*}
  It is easy to see that $r(\cdot)$ is both non-negative and non-negative definite which follows
  immediately from the linearity of the Fourier transform and the fact that the Bessel kernel is
  both non-negative and non-negative definite. Also, we easily deduce from properties (2) - (4) in
  Example~\ref{Ex:Bessel} that
  \begin{align*}
    r(x) \sim
    \begin{cases}
      |x|^{s_1-d} & |x| \to 0, \\
      |x|^{-s_2}  & |x| \to \infty,
    \end{cases}
    \quad \text{and} \quad
    \widehat{r}(\xi) \sim
    \begin{cases}
      |\xi|^{s_2 - d} & |\xi| \to 0, \\
      |\xi|^{-s_1}    & |\xi| \to \infty.
    \end{cases}
  \end{align*}
  Propositions~\ref{P:Bessel} and~\ref{P:Bessel'} imply that
  \begin{align*}
    \Upsilon_{\alpha}(0) =
    \frac{\Gamma\left(\frac{d}{2} -1 + \alpha\right) \Gamma\left(\frac{s_1 -d}{2} + 1 - \alpha\right)}{2^d \pi^{d/2} \: \Gamma(d/2)\:\Gamma\left(s_1/2\right)}
    + \frac{\Gamma\left(1 - \alpha\right) \Gamma\left(\alpha - 1 + \frac{s_2}{2}\right)}{2^{2d} \pi^{3d/2} \: \Gamma(d/2) \: \Gamma\left(s_2 / 2\right)} < \infty,
  \end{align*}
  provided that
	\begin{align*}\label{E:RieszTypeCond}
    0 < d - 2(1-\alpha) < s_1 < d  \qquad \text{and} \qquad
    0 < 2(1-\alpha) < s_2 < d.
	\end{align*}
  In contrast, it is not clear how to compute $\mathcal{F}\big(\sqrt{\widehat{r}}\:\big)$; see
  condition~\eqref{E:TZspec}.
\end{example}

\subsection{Examples of admissible weight functions}\label{SS:Weight}

In this part, we give some examples of the admissible weight functions. As given in Section 2
of~\cite{tessitore.zabczyk:98:invariant}, the following functions are admissible functions:
\begin{equation}\label{E:admRho}
\begin{dcases}
  \rho(x) = \exp(-a|x|)               & a>0, \\
  \rho(x) = \left(1+|x|^a\right)^{-1} & a>d.
\end{dcases}
\end{equation}

The smaller the weight function $\rho(\cdot)$ (not necessarily admissible) is, the larger the space
$L_\rho^2(\R^d)$ is. For example, one may choose $\rho$ to be either a nonnegative function with
compact support or the heat kernel itself $G(1,\cdot)$. In both cases, $\rho$ is smaller than those
in~\eqref{E:admRho} (up to a constant). However, one can easily check that the admissible
condition~\eqref{E:aw} excludes these two cases. However, the examples in the following
Proposition~\ref{P:adm} seem to be less obvious:

\begin{proposition}\label{P:adm}
	$\rho_b(\cdot)$ is admissible if and only if $b\in (0,1]$ where
	\begin{equation*}
	\rho_b(x) \coloneqq \exp\left(-|x|^b\right), \quad x\in\R^d,
	\quad \text{with $b>0$.}
	\end{equation*}
\end{proposition}
\begin{proof}
  From Definition~\eqref{D:Rho}, we see that $\rho_b$ is admissible if and only if for all $T>0$,
	\begin{align*}
	\sup_{(t,x)\in[0,T]\times\R^d} \int_{\R^d} \left(2\pi t\right)^{-d/2} e^{-\frac{|x-y|^2}{2t} + (|x|^b - |y|^b)} \ud y <\infty.
	\end{align*}
  Denote the above integral by $I(t,x)$. We will use $C$ to denote a generic constant that does not
  depend on $(t,x)$, which value may change at each occurrence.

	We first assume that $b\in (0,1]$. In this case,
	\begin{align*}
    | x | ^b
    = | x-y +y | ^b
    \le \left(|x-y| + |y|\right)^b
    \le |x-y|^b + |y|^b.
	\end{align*}
	Hence,
	\begin{align*}
    I(t,x) \le C \int_{\R^d} t^{-d/2} e^{-\frac{|x-y|^2}{2t} + |x-y|^b} \ud y
    = C \int_{\R^d} t^{-d/2} e^{-\frac{|y|^2}{2t} + |y|^b} \ud y
    = C \int_{0}^\infty t^{-d/2} e^{-\frac{r^2}{2t} + r^b} r^{d-1} \ud r.
	\end{align*}
	Then by applying the change of variables $r' = t^{-1/2}r$,
	\begin{align*}
    I(t,x) \le C \int_0^\infty e^{-\frac{r^2}{2} + t^{b/2} r^b} r^{d-1}\ud r
    \le C \int_0^\infty e^{-\frac{r^2}{2} + T^{b/2} r^b} r^{d-1}\ud r
    < \infty.
	\end{align*}

  Next we assume that $b>1$. We need to show that $\rho_b$ is not admissible. Without loss of
  generality, we assume that $d\ge 2$. The case when $d = 1$ is easier and can be proved similarly
  as the proof below. It suffices to show that
	\begin{align*}
    \lim_{r\to\infty} I(1/2,x_r) = \infty, \quad \text{where $x_r:=(r,0,\cdots,0)\in\R^d$.}
	\end{align*}
  Without loss of generality, we may assume below that $r\gg 2$. Denote $y=(y_1,\cdots, y_d) =
  (y_1,y_*)$ with $y_*\in\R^{d-1}$. Using the subadditivity (resp. convexity) of $(x+y)^{b/2}$ when
  $b\in (1,2]$ (resp. $b>2$), we see that
	\begin{align*} 
	| y | ^b =  \left(y_1^2+y_2^2+\cdots+y_d^2\right)^{b/2}
	\le c |y_1|^b + \left(y_2^2+\cdots+y_d^2\right)^{b/2}
	=   c |y_1|^b + c |y_*|^b, \quad  c:= 1 \wedge 2^{b/2 -1}.
	\end{align*}
	Hence,
	\begin{align*}
    I\left(1/2,x_r\right)
    & = C \int_{\R^d} e^{- \left(\sum_{i = 2}^{d} y_i^2\right) - |y_1-r|^2 + (r^b - |y|^b)} \ud y               \\
    & \ge C \int_{\R^{d-1}}\ud y_* \int_\R \ud y_1\: e^{- |y_*|^2 -|y_*|^b - |y_1-r|^2 + (r^b - c |y_1|^b)}     \\
    & = C \int_{\R^{d-1}}\ud y_* \: e^{- |y_*|^2 -c |y_*|^b} \int_\R \ud y_1\: e^{- |y_1-r|^2 + r^b - c|y_1|^b} \\
    & = C \int_\R e^{- y^2 + r^b - c |y-r|^b} \ud y
      \ge C \int_0^r e^{- y^2 + r^b - c |y-r|^b} \ud y \eqqcolon C K(r).
  \end{align*}
	It suffices to show that $\lim_{r\to\infty}K(r) = \infty$, which is true when $b>2$ because
  \begin{align*}
    K(r)
    \ge \int_{r/2}^{r} e^{-y^2 +r^b - c (r-y)^b} \ud y
    \ge \int_{r/2}^{r} e^{-r^2 +r^b - c (r/2)^b} \ud y
    = \frac{r}{2} \exp\left(\left(1-2^{-1-b/2}\right)r^b -r^2\right),
	\end{align*}
	which blows up as $r\to \infty$. Hence, we may assume that $b\in (1,2]$. In this case, $c=1$ and
	\begin{align*} 
	K'(r) = e^{r^b-r^2} + b \int_0^r e^{ - y^2 + r^b - (r-y)^b} \left(r^{b-1}-(r-y)^{b-1}\right) \ud y.
	\end{align*}
  By the intermediate value theorem, we see that $r^{b-1} - (r-y)^{b-1} = (b-1) y \xi^{b-2}$ for
  some $\xi\in [r-y,r]$. Since $b-1\in(0,1]$, this implies that $r^{b-1} - (r-y)^{b-1} \ge (b-1) y
  r^{b-2}$. Hence,
	\begin{align*}
	K'(r) & \ge  e^{r^b-r^2} + b(b-1) \int_1^r e^{ - y^2 + r^b - (r-y)^b } y r^{b-2} \ud y \\
	& \ge  b(b-1)r^{b-2} \int_1^r e^{ - y^2 + r^b - (r-1)^b }  \ud y
	\ge  b(b-1)r^{b-2} e^{r^b-(r-1)^b}\int_1^2 e^{ - y^2 } \ud y.
	\end{align*}
  Another application of the intermediate value theorem shows that $r^b-(r-1)^b = b \xi^{b-1}$ with
  $\xi\in [r-1,r]$. Hence, $r^b-(r-1)^b \ge b (r-1)^{b-1}$ and then
	\begin{align*}
	K'(r) & \ge C r^{b-2} \exp\left(b (r-1)^{b-1}\right).
	\end{align*}
  Hence, for $r \gg 2$, $K'(r)$ is positive and unbounded as $r\to\infty$. Therefore, this implies
  that $K(r)$ blows up as $r\to\infty$, which completes the proof of Proposition~\ref{P:adm}.
\end{proof}

\section*{Acknowledgements}
L.C. thanks Sandra Cerrai for some helpful discussions and for pointing out the
reference~\cite{misiats.stanzhytskyi.ea:16:existence} during the conference -- \textit{Frontier
Probability Days 2018}.

\addcontentsline{toc}{section}{References}
\printbibliography[title = {References}]
\end{document}